\documentclass{article}

\usepackage[T1]{fontenc}
\usepackage[latin1]{inputenc}
\usepackage{epsfig}
\usepackage{amsmath,amssymb,amsthm}
\usepackage{geometry}
\def\Xint#1{\mathchoice
  {\XXint\displaystyle\textstyle{#1}}%
  {\XXint\textstyle\scriptstyle{#1}}%
  {\XXint\scriptstyle\scriptscriptstyle{#1}}%
  {\XXint\scriptscriptstyle\scriptscriptstyle{#1}}%
  \!\int}
\def\XXint#1#2#3{{\setbox0=\hbox{$#1{#2#3}{\int}$}
    \vcenter{\hbox{$#2#3$}}\kern-.5\wd0}}

\def\fint{\Xint-}
\def\longrightharpoonup{\DOTSB\relbar\joinrel\rightharpoonup}

\def\NN{\mathbb N}

\def\RR{\mathbb R}
\def\ZZ{\mathbb Z}

\def\div{\operatorname{div}}
\def\diam{\operatorname{diam}}
\begin{document}
\numberwithin{equation}{section}
\newtheorem{theoreme}{Theorem}[section]
\newtheorem{proposition}[theoreme]{Proposition}
\newtheorem{remarque}[theoreme]{Remark}
\newtheorem{lemme}[theoreme]{Lemma}
\newtheorem{corollaire}[theoreme]{Corollary}
\newtheorem{definition}[theoreme]{Definition}
\newtheorem{hypothese}{Assumption}
\renewcommand{\thehypothese}{(A\arabic{hypothese})}

\title{Precised approximations in elliptic homogenization beyond the periodic setting}
\author{X. Blanc$^1$, M. Josien$^2$ \& C. Le Bris$^2$\\
  {\footnotesize $^1$ Université Paris-Diderot, Sorbonne Paris-Cité, Sorbonne Université,}\\
  {\footnotesize CNRS, Laboratoire Jacques-Louis Lions, F-75013 Paris.}\\
 {\footnotesize{\tt blanc@ann.jussieu.fr}}\\
{\footnotesize $^2$ Ecole des Ponts and INRIA,}\\
{\footnotesize 6 \& 8, avenue Blaise Pascal, 77455 Marne-La-Vall\'ee
  Cedex 2, FRANCE}\\
{\footnotesize{\tt claude.le-bris@enpc.fr, marc.josien@enpc.fr}}
}

\maketitle
\begin{abstract}
We consider homogenization problems for linear elliptic equations in divergence form. 
The coefficients are assumed to be a local perturbation of some periodic background. We prove $W^{1,p}$ and
Lipschitz convergence of the two-scale expansion, with explicit rates. For this purpose, we use a corrector adapted
to this particular setting, and defined in \cite{BLLcpde,BLLfutur1}, and apply the same strategy of proof as
Avellaneda and Lin in \cite{AvellanedaLin}. We also
propose an abstract setting generalizing our particular assumptions for which the same estimates hold. 
\end{abstract}
\tableofcontents

\section{Introduction}

The present paper follows up on the articles \cite{BLLMilan,BLLcpde,BLLfutur1,BLLfutur2}. In these works, we
studied homogenization theory for linear elliptic equations, for which the coefficients are assumed to be
periodic and perturbed by local defects, that is, $L^r(\RR^d)$ functions, $r\in ]1,+\infty[$. As expected, the
macroscopic behavior, in the homogenization limit, is defined by the periodic background only. However, if one is
interested in finer convergence properties, possibly with a convergence rate, then the defect may have an
impact. In such a case, a corrector taking the defect into account is necessary. Its existence has been proved in
\cite{BLLMilan} in the case $r=2$, and in \cite{BLLcpde,BLLfutur1} in the general case. Formal arguments in
\cite{BLLMilan,BLLcpde} indicate that this adapted corrector is important for having a good convergence rate
and/or convergence in a finer topology.

The aim of the present paper is to prove that the corrector constructed in
\cite{BLLMilan,BLLcpde,BLLfutur1,BLLfutur2} indeed allows for such convergence results. The  work \cite{AvellanedaLin} and, more recently, \cite{KLSGreenNeumann}, are the two major reference works on these issues. They both address the periodic setting, and we will briefly summarize the important results they established in Section~\ref{sec:periodic-case} below. 

Our proofs, in the setting of a periodic geometry perturbed by a local defect, closely follow the general pattern
of the proofs exposed in \cite{KLSGreenNeumann} and reproduce many key ingredients and details of both
\cite{AvellanedaLin} and \cite{KLSGreenNeumann}. For the sake of clarity and brevity, and also with a specific
pedagogic purpose because the arguments may become very rapidly technical,  we have however decided to present our
proofs in the particular case of \emph{equations}, as opposed to \emph{systems}. Some simplifications of the proofs
of \cite{AvellanedaLin,KLSGreenNeumann}, which all apply to systems as well as to equations, are then possible. The
reader might better, then, appreciate the string of key arguments, in the absence of some unnecessary
technicalities. Similarly, we have also provided some additional internal details of the proofs which can be useful
to non experts for a better comprehension. Our results carry over
to elliptic systems (satisfying the Legendre condition, as is the case for \cite{AvellanedaLin,KLSGreenNeumann}),
provided some of the arguments are adjusted, and then follow those of \cite{AvellanedaLin,KLSGreenNeumann} even
more closely. We did not check all the details in this direction.


One interesting feature we emphasize in the present contribution is that the results of \cite{AvellanedaLin,KLSGreenNeumann} of the periodic setting indeed carry over not only to the \emph{perturbed} periodic setting, but also to a quite general abstract setting, which we make precise in Section~\ref{sec:periodic-case-with} below. The latter observation about the generalization of the results of \cite{AvellanedaLin} and related works to non periodic setting is corroborated by the recent works \cite{bella-giunti-otto,gloria-neukamm-otto-2014}. Some  of the necessary assumptions presented there (in the context of random homogenization) are quite close in spirit to our own formalization.


We consider the following problem:
 \begin{equation}
  \label{eq:equation}
  \left\{
  \begin{aligned}
  -\div\, \left[a\left(\frac x \varepsilon\right)\,\nabla u^\varepsilon\right]=f && \text{in} && \Omega,\\
  u^\varepsilon = 0 &&\text{on} && \partial\Omega.
  \end{aligned}\right.
\end{equation}
Here, $\Omega$ is a domain of $\RR^d$, the regularity of which will be made precise below. The right-hand side $f$
is in $ L^q(\Omega)$ for some $q\in ]1,+\infty[$, and the matrix-valued
coefficient $a$ satisfies the following assumptions:
\begin{equation}
  \label{eq:aper+tildea}
  a=a^{per}+\widetilde a
\end{equation}
where $a^{per}$ denotes a \emph{periodic} unperturbed background, and $\widetilde a$ the perturbation, with 
\begin{equation}
\label{eq:hyp1}
\left\{
\begin{array}{l}
a^{per}(x)+\widetilde a(x) \quad \hbox{\rm and}\quad a^{per}(x) \quad \hbox{\rm are both uniformly elliptic, in}\,x\in\RR^d, 
\\
a^{per}\in \left(L^\infty(\RR^d)\right)^{d\times d},
\\
  \widetilde a \in \left(L^\infty(\RR^d)\cap L^r(\RR^d)\right)^{d\times d}, \quad \hbox{\rm for some }\quad 1\leq r<+\infty
\\
   a^{per},\, \widetilde a\in \left(C^{0,\alpha}_{\rm unif}\left(\RR^d\right)\right)^{d\times d} \quad\hbox{\rm for some}\quad \alpha>0,
\end{array}
\right.
\end{equation}
where $C^{0,\alpha}_{\rm unif}(\RR^d)$ denotes the space functions that are, uniformly on $\RR^d$, H\"older
continuous with coefficient $\alpha$.

From now on, we will not make the distinction between the spaces $L^q(\Omega),$ $\left(L^q(\Omega)\right)^d$ and
$\left(L^q(\Omega)\right)^{d\times d}$, denoting $\|z\|_{L^q(\Omega)}$ the norm of $z$ even if $z$ is a
vector-valued or a matrix-valued function. The same convention is adopted for Hölder spaces $C^{0,\alpha}$.

We also note that we assume $d\geq 3$. All our proofs and results can be adapted to the dimension $d=2$. Of
course, dimension $1$ is specific and can be addressed by (mostly explicit) analytic arguments that we
omit here.

\medskip

All the results we present here have been announced in \cite{crasBJL}, and are part of the PhD thesis
\cite{theseMJ}.

\subsection{The periodic case}\label{sec:periodic-case}

In the periodic case, that is, $\widetilde a \equiv 0$, it is well-known (see for instance \cite{BLP}) that problem \eqref{eq:equation} converges to the
following homogenized problem
 \begin{equation}
  \label{eq:homog_periodique}
  \left\{
  \begin{aligned}
  -\div\, \left[a^*\,\nabla u^*\right]=f && \text{in} && \Omega,\\
  u^* = 0 &&\text{on} && \partial\Omega,
  \end{aligned}\right.
\end{equation}
where $a^*$ is a constant matrix. It is classical that $u^\varepsilon\longrightarrow u^*$ in $L^2(\Omega)$, and that
$\nabla u^\varepsilon \longrightharpoonup \nabla u^*$ in $L^2(\Omega)$. In order to have strong convergence of the
gradient, correctors need to be introduced, that is, the solutions to the following problem
\begin{equation}
\label{eq:correcteur-per}
  -\div\left(a^{per}(x)\,(p+\nabla w_{p,per}(x))\right)=0,\quad w_{p,per}\text{ is periodic,}
\end{equation}
posed for each fixed vector $p\in \RR^d$. It is well-known (see, here again, \cite{BLP}), that problem
\eqref{eq:correcteur-per} has a unique solution (up to the addition of a constant), for any $p\in\RR^d$.
Given (\ref{eq:hyp1}), elliptic regularity implies that $\nabla w_{p,per}\in C^{0,\alpha}_{\rm unif}(\RR^d)$. Introducing the remainder
\begin{equation}
  \label{eq:reste-per}
  R^\varepsilon_{per}(x):= u^\varepsilon(x) - u^*(x) -\varepsilon\sum_{j=1}^d w_{per,e_j}\left(\frac
    x\varepsilon\right) \partial_j u^*(x),
\end{equation}
the results of \cite{BLP} imply that $\nabla R^\varepsilon_{per}\longrightarrow 0$ in $L^2(\Omega)$, with the following
convergence rate:
\begin{equation}
  \label{eq:cv_BLP}
  \left\|\nabla R^\varepsilon_{per}\right\|_{L^2(\Omega)}\leq C\|f\|_{L^2(\Omega)} \, \sqrt\varepsilon,
\end{equation}
for some constant independent of $f$. The convergence rate $O(\sqrt\varepsilon)$ is mainly due to the existence of
a boundary layer,
and an $O(\varepsilon)$ convergence can actually be proved for interior domains. 

In \cite{griso-2002,griso-2004}, the generalization of the above results (both \eqref{eq:cv_BLP} and interior convergence of
order $\varepsilon$) are proved under more general assumptions ($\Omega$ of class $C^{1,1}$,
$a\in L^\infty$ and
the corrector is not assumed to have its gradient in $L^\infty$). Also in \cite{griso-2004}, in the case of Lipschitz domains, a
convergence up to the boundary of order $\varepsilon^{\gamma}$, for some $0<\gamma\leq 1/3$, is established. 

In order to have a $O(\varepsilon)$ convergence rate up to the boundary, an
adaptation of the corrector is needed. This question was studied in
\cite{onofrei-vernescu} in the case of non-H\"older coefficients. For the case of systems (as opposed to equations)
it was studied in \cite{KLS2012} (actually also with non-homogeneous Dirichlet conditions).

Issues regarding the convergence of the remainder are also addressed in \cite{AvellanedaLin}, where Avellaneda and Lin proved uniform (with respect to $\varepsilon$) continuity for the
operator $L_\varepsilon$ which, to the couple $(f,g)$, associates the solution $u^\varepsilon$ of
\eqref{eq:equation} with Dirichlet condition $u^\varepsilon = g$. This continuity holds from $L^q(\Omega)\times
C^{0,\gamma}(\partial\Omega)$ to $C^{0,\mu}(\Omega)$ if $q>d$, with $\mu = \min(\gamma,d/q)$. If $q\leq d$, with
homogeneous Dirichlet conditions, the continuity holds from $L^q(\Omega)$ to $W^{1,r}(\Omega),$ with $1/r+1/q =
1/d$. These results also hold for systems, and actually improve an earlier and more restricted work \cite{AL-1987}. In \cite{AL-1989-2}, the same kind of
results were extended to equations in non-divergence form. In \cite{KLSGreenNeumann}, estimates were proved
for the convergence of the Green functions associated to \eqref{eq:equation}, both for Dirichlet and Neumann
conditions. These estimates allow to prove the convergence rate of $R^\varepsilon$ in $W^{1,p}$. 

All the above results are valid only for periodic coefficients. 
In the preprint \cite{gloria-neukamm-otto-2014}, some important results of \cite{AvellanedaLin}
were extended to the stochastic case, using the idea that, in \cite{AvellanedaLin}, periodicity was only used to
ensure some uniform $H$-convergence. This is also a key idea of the present work. 

\subsection{The periodic case with a local defect}\label{sec:periodic-case-with}

In order to develop the approximation estimates for (\ref{eq:equation})-(\ref{eq:aper+tildea})-(\ref{eq:hyp1})
for $\widetilde a \not\equiv 0$, we define the corrector problem 
\begin{equation}
  \label{eq:correcteur}
\left\{
  \begin{aligned}
      &-\,\div\left(a\,(p+\nabla w_{p})\right)=0 \quad  \text{in } \RR^d, \\
      &\frac{|w_p(x)|}{1+|x|} \mathop{\longrightarrow}_{|x|\to+\infty} 0.
  \end{aligned}
\right.
\end{equation}
In the special case $\widetilde a \equiv 0$, a Liouville-type theorem was proven in \cite{AL-1989-CRAS}, showing that \eqref{eq:correcteur} reduces to
\eqref{eq:correcteur-per}: up to the addition of a constant, the only solution that is strictly sublinear at infinity is the periodic solution. In
the case $\widetilde a\not\equiv 0$, it has been proven in \cite{BLLMilan,BLLcpde,BLLfutur1} (see also the recent
work \cite{JiangLin}, that brings a different perspective) that Problem~\eqref{eq:correcteur} has a
solution, that reads as
\begin{equation}
  \label{eq:correcteur-decomposition}
  w_p = w_{p,per} + \widetilde w_p, 
\end{equation}
where $w_{p,per}$ is the solution to \eqref{eq:correcteur-per}, $\widetilde w_p$ is the solution to 
\medskip
\begin{equation}
  \label{eq:correcteur-tilde}
  -\,\div\left(a\,\nabla \widetilde w_{p}\right)= \,\hbox{\rm div}\left(\widetilde a\,(p+\nabla w_{p,per})\right)\quad \text{in } \RR^d,
\end{equation}
and, if $\widetilde a \in L^r(\RR^d)$, $\nabla \widetilde w_p\in
L^r(\RR^d)$, for any $r\in ]1,+\infty[$. Even if $\widetilde a\not\equiv 0$, the proofs of \cite{BLP} still imply in
this case that
$u^\varepsilon\longrightarrow u^*$ in $L^2(\Omega)$ and $\nabla u_\varepsilon \longrightharpoonup \nabla u^*$ in $L^2(\Omega)$,
as $\varepsilon\to 0$, where $u^*$ solves \eqref{eq:homog_periodique}, and the matrix $a^*$ is equal to the
periodic homogenized matrix. However, in order to improve and quantify this convergence, \cite{BLLMilan,BLLcpde,BLLfutur1} show
that we need to replace the periodic corrector (\ref{eq:correcteur-per}) by the solution to
\eqref{eq:correcteur}, and define
\begin{equation}
  \label{eq:reste}
  R^\varepsilon(x):= u^\varepsilon(x) - u^*(x) -\varepsilon\sum_{j=1}^d w_{e_j}\left(\frac x\varepsilon\right) \partial_j u^*(x),
\end{equation}
instead of (\ref{eq:reste-per}). Then we have:
\begin{theoreme}[\textbf{Local defects in periodic backgrounds}]\label{th:defaut} Assume $d\geq 3$.
  Consider \eqref{eq:equation}, where the matrix-valued coefficient $a$ satisfies \eqref{eq:aper+tildea}, and $a^{per}$
  and $\widetilde a$ satisfy \eqref{eq:hyp1}. Assume that $\Omega$ is a $C^{2,1}$ domain, that $\Omega_1
  \subset\subset \Omega$, that $r\neq d$ and define
  \begin{equation}
    \label{eq:nu_r}
    \nu_r = \min \left(1,\frac d r \right) \in ]0,1].
  \end{equation}
Let $f\in L^2(\Omega)$, and let $u^\varepsilon$, $u^*$ be the solutions to \eqref{eq:equation} and
\eqref{eq:homog_periodique}, respectively. Define $R^\varepsilon$ by \eqref{eq:reste}, where the corrector $w_p$
with $p=e_j$, $1\leq j\leq d$, is
defined by \eqref{eq:correcteur-decomposition}-\eqref{eq:correcteur-tilde}-\eqref{eq:correcteur-per} (thus in
particular solves \eqref{eq:correcteur}). Then $R^\varepsilon$
satisfies the following:
\begin{enumerate}
\item $R^\varepsilon\in H^1(\Omega)$, and
  \begin{equation}
    \label{eq:cv_L2}
    \left\|R^\varepsilon\right\|_{L^2(\Omega)} \leq C \varepsilon^{\nu_r} \|f\|_{L^2(\Omega)}, 
  \end{equation}
\begin{equation}
  \label{eq:cv_H1}
    \left\|\nabla R^\varepsilon\right\|_{L^2(\Omega_1)} \leq C \varepsilon^{\nu_r} \|f\|_{L^2(\Omega)}.
\end{equation}
\item If $f\in L^q(\Omega)$ for some $q\geq 2$, then $R^\varepsilon\in W^{1,q}(\Omega)$ and
\begin{equation}
  \label{eq:cv_W1p}
    \left\|\nabla R^\varepsilon\right\|_{L^q(\Omega_1)} \leq C \varepsilon^{\nu_r} \|f\|_{L^q(\Omega)}.
\end{equation}
\item If $f\in C^{0,\beta}(\overline \Omega)$ for some $\beta\in ]0,1[$, then $R^\varepsilon\in
  W^{1,\infty}(\Omega)$ and
\begin{equation}
  \label{eq:cv_lipschitz}
    \left\|\nabla R^\varepsilon\right\|_{L^\infty(\Omega_1)} \leq C \varepsilon^{\nu_r}\ln\left(2+\varepsilon^{-1}\right) \|f\|_{C^{0,\beta}(\Omega)},
\end{equation}
\end{enumerate}
where, in (\ref{eq:cv_L2}) through (\ref{eq:cv_lipschitz}), the various constants $C>0$ do not depend on $f$ nor on
$\varepsilon$. 
\end{theoreme}

Given (\ref{eq:nu_r}), this result gives two different behaviors of the remainder $R^\varepsilon$ according to $r<d$ or $r>d$. In the first
case, the defect is so localized that the estimates are exactly those of the periodic case \cite{KLSGreenNeumann}. On the contrary, if
$r>d$, the defect is spread out, and the quality of approximation deteriorates as $r$ grows. In the
critical case $r=d$, we can apply the results of the case $r>d$ in order to have the above
estimates, in which $\varepsilon^{\nu_r}$ is replaced by $\varepsilon^\nu$, for any $\nu<\nu_r=1$.

As already pointed out in \cite{BLLcpde}, the case $r=d$ is a critical case for the existence of a
corrector. Indeed, even if $a^{per}\equiv 1$, hence $\nabla w^{per} = 0$, the corrector equation reads as
\begin{displaymath}
  -\Delta \widetilde w_p = \div\left(\widetilde a p\right).
\end{displaymath}
Hence, as $|x|\to +\infty$, $\displaystyle\widetilde w_{p}(x)\approx C\int{{(x-y)\cdot \left(\widetilde
      a(y)\,p\right)}\over{|x-y|^{d}}}dy,$ for some constant $C\neq 0$. This
makes clear the fact that $\widetilde a(y) \sim |y|^{-1}$ is reminiscent of the criticality of the space $L^d(\RR^d)$.

\begin{remarque}
  \label{rq:bord}
In Theorem~\ref{th:defaut}, the domain $\Omega$ is assumed to be $C^{2,1}$. However, as
far as estimates \eqref{eq:cv_L2}-\eqref{eq:cv_H1}-\eqref{eq:cv_W1p} are concerned, a $C^{1,1}$ regularity is
sufficient. The regularity $C^{2,1}$ is only necessary to prove that $R^\varepsilon\in
W^{1,\infty}(\Omega)$. 
\end{remarque}

\subsection{Abstract general assumptions}
\label{sec:generalisation}

As we shall see below, Theorem~\ref{th:defaut} is a consequence of a more general, abstract, result that we
state in the present subsection. The point is that, in the theory of \cite{AvellanedaLin}, the periodicity of
the matrix-valued coefficient $a$ is essentially useful in order to have a bounded corrector. This assumption may be replaced by
uniform $H$-convergence (a notion which is made precise below in Definition~\ref{def:uniform-H-cv}). 

Let us now emphasize that (\ref{eq:equation}) considers a \emph{rescaled} coefficient $a\left(\frac x
  \varepsilon\right)$, which is a strong assumption of our setting. This implies, since $a^*$ is defined as some
weak limit of functions of $x/\varepsilon$, that $a^*$ is homogeneous of degree $0$. Hence, if it is continuous, it
must be a constant. This is why we hereafter assume that
\begin{equation}
  \label{eq:a_constante}
  a^* \quad\text{is a constant matrix.}
\end{equation}

We now introduce a set of assumptions that
formalize our mathematical setting. We consider a matrix-valued coefficient $a$ that satisfies the following conditions
\begin{hypothese}
  \label{H1}
  There exists $\mu>0$ such that
  \begin{displaymath}
    \forall x\in\RR^d, \ \forall \xi\in\RR^d, \quad \mu |\xi|^2 \leq
  \left(a(x)\xi\right)\cdot\xi \leq \frac 1 \mu |\xi|^2.
  \end{displaymath}
\end{hypothese}
\begin{hypothese}
  \label{H2}
There exists $\alpha\in ]0,1[$ such that $a\in C^{0,\alpha}_{\rm unif}(\RR^d).$
\end{hypothese}
Assumptions~\ref{H1} and \ref{H2} are standard, and were made already in \cite{AvellanedaLin}. We now give
more specific assumptions that aim at
generalizing periodicity. The first one is the existence of a corrector:
\begin{hypothese}
  \label{H3}
For any $p\in\RR^d$, there exists $w_p\in H^1_{\rm loc}(\RR^d)$ solution to the corrector equation \eqref{eq:correcteur}.
\end{hypothese}
As in the periodic case, we assume that the gradient of the corrector is bounded uniformly:
\begin{hypothese}
  \label{H4}
For any $p\in\RR^d$, the gradient of $w_p$ is in $L^2_{\rm unif}(\RR^d)$, that is:
\begin{displaymath}
  \|\nabla w_p\|_{L^2_{\rm unif}(\RR^d)} := \sup_{x\in\RR^d} \|\nabla w_p\|_{L^2(B(x,1))} <+\infty,
\end{displaymath}
where $B(x,1)$ denotes the unit ball of center $x$. 
\end{hypothese}
In the periodic case, we have $\nabla w_p\left(\frac \cdot \varepsilon\right) \longrightharpoonup 0$ as
$\varepsilon\to 0$. Moreover, this property is uniform with respect to translation. This is a property we will
\emph{impose} here:
\begin{hypothese}
  \label{H5}
For any sequence $(y_n)_{n\in\NN}$ of vectors in $\RR^d$ and any sequence $\varepsilon_n \to 0$, and for any $p\in
\RR^d$, 
\begin{displaymath}
  \int_Q \nabla w_p\left(\frac x {\varepsilon_n} +  y_n\right)dx \mathop{\longrightarrow}_{n\to+\infty} 0,
\end{displaymath}
where $Q$ is the unit cube of $\RR^d$. 
\end{hypothese}
With a view to addressing non-symmetric matrix-valued coefficients, note that, in contrast to \eqref{eq:aper+tildea}, the fact that $a$ satisfies Assumption~\ref{H3}-\ref{H4}-\ref{H5} does not imply that $a^T$ does. We will in some
situations need to assume that $a^T$ \emph{also} satisfies Assumption~\ref{H3}-\ref{H4}-\ref{H5}, and likewise other assumptions that
follow below. In such a case, we denote by $w^T_p$ the
corrector associated to the coefficient $a^T$. 

\medskip

We will assume that the convergence to the homogenized matrix $a^*$ is uniform in the following sense:
\begin{hypothese}
  \label{H6}
There exists a \emph{constant} matrix $a^*$ such that, for any sequence $(y_n)_{n\in\NN}$ of vectors in $\RR^d$, any
sequence $\varepsilon_n \to 0$ and for any $p\in \RR^d$,
\begin{displaymath}
  \int_Q a\left(\frac x {\varepsilon_n} +  y_n\right)\left(p+ \nabla w_p\left(\frac x {\varepsilon_n} +
      y_n\right)\right) dx \mathop{\longrightarrow}_{n\to+\infty} a^* p,
\end{displaymath}
where the matrix $a^*$ is the homogenized matrix in \eqref{eq:homog_periodique}.
\end{hypothese}
It is stated in Proposition~\ref{prop:uniform-H-cv} below that this
implies \emph{uniform H-convergence,} in the sense of the following definition:
\begin{definition}\label{def:uniform-H-cv}
  We say that the matrix-valued coefficient $a\left(\frac x \varepsilon\right)$ uniformly
  H-converges to $a^*$ if for any sequence $\varepsilon_n\to 0$ and any
  sequence $\left(y_n\right)_{n\in\NN},$
\begin{displaymath}
  a\left(\frac x {\varepsilon_n}+y_n\right)\quad\text{H-converges to}
  \quad a^*.
\end{displaymath}
\end{definition}
For the definition of H-convergence itself, we refer to \cite[Definition
  1]{Murat1997} or \cite[Definition 6.4]{Tartar}.

\medskip

As we will see below, an important quantity in order to analyze the behaviour of the remainder $R^\varepsilon$
defined by \eqref{eq:reste} is the potential associated with $a$. In order to define it, we first introduce the
vector field $M_k$ defined by 
\begin{equation}
  \label{eq:Mik}
  M_k^i(x) = a_{ik}^* - \sum_{j=1}^d a_{ij}(x) \left(\delta_{jk} + \partial_j w_{e_k}(x)\right), \quad 1\leq i \leq d,
\end{equation}
which is divergence-free, according to \eqref{eq:correcteur}. Hence, \emph{formally}, there exists $B_k^{ij}(x)$, which is
skew-symmetric with respect to the indices $i,j$, and is solution to $\div(B_k) = M_k$, that is,
\begin{equation}
  \label{eq:B-antisymetrique}
  \forall i,j,k\in \{1,\dots,d\}, \quad  B_k^{ij} = - B_k^{ji}.
\end{equation}
\begin{equation}
  \label{eq:B-equation}
  \forall j,k\in\left\{1, \dots , d\right\}, \quad \sum_{i=1}^d \partial_i B_k^{ij} = M_k^j.
\end{equation}
A simple way to build this potential $B$ is to solve the following equation
\begin{equation}
  \label{eq:B-equation-2}
  \forall \, i,j,k\in\left\{1, \dots , d\right\}, \quad-\Delta B^{ij}_k = \partial_j M_k^i - \partial_i M_k^j.
\end{equation}
It is clear that if $B$ solves (\ref{eq:B-equation-2}), then it satisfies (\ref{eq:B-antisymetrique}). Moreover,
taking the divergence of (\ref{eq:B-equation-2}), we get $-\Delta (\div(B)) = -\Delta M$, that is,
\begin{displaymath}
  \forall j,k\in\left\{1, \dots , d\right\}, \quad-\Delta \left(\sum_{i=1}^d\partial_i B^{ij}_k\right) = -\Delta M^j_k.
\end{displaymath}
Hence, up to the addition of a harmonic function, we find (\ref{eq:B-equation}). In most cases, this harmonic
function is necessarily a constant (think for instance of the periodic case). 

The above construction can be made precise in the periodic case  (see \cite{JKO}, pp 26-27). We will see below how
and why the construction also makes sense in our setting (\ref{eq:aper+tildea})-(\ref{eq:hyp1}).

The link between $B$ and $R^\varepsilon$ will be clear below when we write the equation satisfied by
$R^\varepsilon$ (see \eqref{eq:equation-reste}-\eqref{eq:H-epsilon}). In order to apply a method close to that of
\cite{AvellanedaLin}, we are going to assume that, in some sense, $\varepsilon w_p(x/\varepsilon)$ and $\varepsilon
B(x/\varepsilon)$ vanish as $\varepsilon\to 0$. This is the meaning of the following two assumptions
\begin{hypothese}
  \label{H7}
There exists $C>0$ and $\nu\in [0,1[$ such that, for any $x\in \RR^d$, any $y\in\RR^d$, and any $k\in \{1,\dots, d\}$,
\begin{displaymath}
  |x-y|\geq 1 \ \Rightarrow \ \left|w_{e_k}(x)-w_{e_k}(y)\right|\leq C|x-y|^{1-\nu}.
\end{displaymath}
\end{hypothese}
\begin{hypothese}
  \label{H8}
There exists a potential $B\in H^1_{\rm loc}(\RR^d)$ defined by \eqref{eq:B-equation-2}, and there exists $C>0$
such that, for any $x\in\RR^d$ and any $y\in\RR^d$, 
\begin{displaymath}
  |x-y|\geq 1 \ \Rightarrow \ \left|B(x)-B(y)\right|\leq C|x-y|^{1-\nu}.
\end{displaymath}
Here, the constant $\nu\in [0,1[$ is assumed to be the same as in Assumption~\ref{H7}.
\end{hypothese}

Proposition~\ref{pr:3} below will establish that, in the case of a coefficient $a$ satisfying \eqref{eq:aper+tildea}
and \eqref{eq:hyp1}, the above assumptions are satisfied with $\nu = \nu_r$ defined by \eqref{eq:nu_r}.

\medskip

Our main result in this general abstract setting is
\begin{theoreme}[\textbf{Abstract general setting}]\label{th:general}
  Assume $d\geq 3$ and that the coefficients $a$ and $a^T$ (and their respective correctors $w_p$ and $w_p^T$)
  satisfy Assumptions \ref{H1} through \ref{H6}, and \ref{H7}-\ref{H8} for some $\nu>0$. Assume that $\Omega$ is a $C^{2,1}$ domain,
  and that $\Omega_1\subset\subset \Omega$. Let $\varepsilon\in ]0,1[$ and let $u^\varepsilon$, $u^*$,
  $R^\varepsilon$ be defined by \eqref{eq:equation}, \eqref{eq:homog_periodique}, \eqref{eq:reste}, respectively,
  where $f\in L^2(\Omega)$. Then we have 
\begin{enumerate}
\item $R^\varepsilon\in H^1(\Omega)$, and
  \begin{equation}
    \label{eq:cv_L2_general}
    \left\|R^\varepsilon\right\|_{L^2(\Omega)} \leq C \varepsilon^{\nu} \|f\|_{L^2(\Omega)},
  \end{equation}
\begin{equation}
  \label{eq:cv_H1_general}
    \left\|\nabla R^\varepsilon\right\|_{L^2(\Omega_1)} \leq C \varepsilon^{\nu} \|f\|_{L^2(\Omega)}.
\end{equation}
\item If $f\in L^p(\Omega)$ for some $p\geq 2$, then $R^\varepsilon\in W^{1,p}(\Omega)$ and
\begin{equation}
  \label{eq:cv_W1p_general}
    \left\|\nabla R^\varepsilon\right\|_{L^p(\Omega_1)} \leq C \varepsilon^{\nu} \|f\|_{L^p(\Omega)}.  
\end{equation}
\item If $f\in C^{0,\beta}(\overline \Omega)$ for some $\beta\in ]0,1[$, then $R^\varepsilon\in
  W^{1,\infty}(\Omega)$ and
\begin{equation}
  \label{eq:cv_lipschitz_general}
    \left\|\nabla R^\varepsilon\right\|_{L^\infty(\Omega_1)} \leq C \varepsilon^{\nu}\ln\left(2+\varepsilon^{-1}\right) \|f\|_{C^{0,\beta}(\Omega)},
\end{equation}
\end{enumerate}
where in (\ref{eq:cv_L2_general}) through (\ref{eq:cv_lipschitz_general}), the various constants $C>0$ do not depend on
$f$ nor on $\varepsilon$.
\end{theoreme}

The proof of Theorem~\ref{th:general} will consist in applying the strategy of proof of \cite{AvellanedaLin} and
\cite{KLSGreenNeumann}, which were originally restricted to the periodic case. Here, periodicity is replaced by
Assumptions~\ref{H3} through \ref{H8}. The proofs follow those of
\cite{AvellanedaLin,KLSGreenNeumann}, but we need to everywhere keep track of the use of assumptions \ref{H3} through
\ref{H8}, and check that these properties are sufficient to proceed at each step of the arguments. 

\begin{remarque}
  \label{rq:bord-general}
  As we already pointed out in Remark~\ref{rq:bord} for the specific case of localized defects, in
  Theorem~\ref{th:general}, the assumption that $\Omega$ is of class $C^{2,1}$ is, here again, only needed for the estimate $\|\nabla R^\varepsilon\|_{L^{\infty}(\Omega)}$.
\end{remarque}

Given this result, it is clear that proving Theorem~\ref{th:defaut} amounts to proving that, in the case of a
defect, Assumptions \ref{H1} through \ref{H8} are satisfied with $\nu = \nu_r$ defined by \eqref{eq:nu_r}. 

\bigskip

Our article is organized as follows. In Section~\ref{sec:preliminaries}, we start with some comments on 
Assumptions \ref{H1} through \ref{H8}. Then we study the existence and uniqueness of the potential $B$, and we
relate it to the remainder $R^\varepsilon$, using \eqref{eq:equation-reste}-(\ref{eq:H-epsilon}), that is,
\begin{displaymath}
  -\div\left(a\left(\frac x \varepsilon\right) \nabla R^\varepsilon\right) = \div\left(H^\varepsilon\right),
\end{displaymath}
with
\begin{displaymath}
  H^\varepsilon_i(x) = \varepsilon \sum_{j,k=1}^d a_{ij}\left(\frac{x}{\varepsilon}\right) w_{e_k}\left( \frac{x}{\varepsilon}\right) \partial_j\partial_ku^*(x)
    -\varepsilon \sum_{j,k=1}^d  B^{ij}_k\left(\frac{x}{\varepsilon}\right)\partial_j\partial_k u^*(x),
\end{displaymath}
Our method to prove estimates on $R^\varepsilon$ relies on some regularity properties of the operator
$-\div(a(x/\varepsilon)\nabla \cdot)$ on the one hand, and bounds on the right-hand side $H^\varepsilon$ on the
other hand. In Section~\ref{sec:homogeneous-case}, we prove
such regularity estimates in the homogeneous case (that is, if the right-hand side is $0$). In Section~\ref{sec:inhomogeneous-case}, we extend these
results to the inhomogeneous case. Finally, in Section~\ref{sec:proof}, we conclude the proof of
Theorem~\ref{th:general} (abstract setting) and that of Theorem~\ref{th:defaut} (local defects).

\section{Preliminaries}
\label{sec:preliminaries}
\subsection{Some remarks on our assumptions}

\paragraph{Alternative formulations of our Assumptions.}

Assume \ref{H1}, \ref{H2} and \ref{H3}. Then, it is clear that Assumptions~\ref{H4} and
\ref{H5} are equivalent to
\begin{displaymath}
  \nabla w_p \left(\frac x {\varepsilon_n} + y_n\right) \mathop{\longrightharpoonup}_{n\to+\infty} 0
  \quad\text{in}\quad L^2(\mathcal D), 
\end{displaymath}
for any bounded Lipschitz domain $\mathcal D$, any $p\in\RR^d$, and for any sequences $\left(y_n\right)_{n\in\NN}$ and $\varepsilon_n\to 0$. 

Similarly, if Assumptions~\ref{H1}, \ref{H2} and \ref{H3} are satisfied, Assumptions \ref{H4} and
\ref{H6} are equivalent to
\begin{displaymath}
  a\left(\frac x{\varepsilon_n}+ y_n\right)\left(p+\nabla w_p \left(\frac x {\varepsilon_n} + y_n\right)\right)
  \mathop{\longrightharpoonup}_{n\to+\infty} a^*p \quad\text{in}\quad L^2(\mathcal D), 
\end{displaymath}
for any bounded Lipschitz domain $\Omega$, any $p\in\RR^d$, and for any sequences $\left(y_n\right)_{n\in\NN}$ and $\varepsilon_n\to 0$. 

\medskip

Another important point is that Assumptions~\ref{H4} and \ref{H5} are in fact equivalent to some strict sublinearity
condition at infinity for the corrector:
\begin{lemme}\label{lm:rq-sous-linearite}
  Assume that the matrix-valued coefficient $a$ satisfies Assumptions~\ref{H1} and \ref{H3}. Then, it satisfies Assumptions~\ref{H4} and \ref{H5}
  if and only if
  \begin{equation}
    \label{eq:ss-linearite1}
    \forall p\in \RR^d, \quad \lim_{|x|\to +\infty} \left(\sup_{y\in\RR^d} \frac{|w_p(x+y)-w_p(y)|}{1+|x|}\right) = 0.
  \end{equation}
\end{lemme}
\begin{proof} We first assume that Assumptions~\ref{H4} and \ref{H5} are satisfied and prove
  \eqref{eq:ss-linearite1} using a contradiction argument. If \eqref{eq:ss-linearite1} does not hold, then there
  exists two sequences $(y_n)_{n\in\NN}$ and $(x_n)_{n\in\NN}$ such that
  \begin{displaymath}
    |x_n|\mathop{\longrightarrow}_{n\to+\infty} + \infty, \quad \text{and}\quad \frac{|w_p(x_n+y_n)-w_p(y_n)|}{1+|x_n|}
    \geq \gamma >0,
  \end{displaymath}
where $\gamma$ does not depend on $n$. Defining $\varepsilon_n = |x_n|^{-1}$ and $\overline x_n = x_n/|x_n|$, this
inequality implies
\begin{displaymath}
  \varepsilon_n \to 0, \quad |\overline x_n| = 1, \quad \varepsilon_n\left|w_p\left(\frac{\overline x_n}{\varepsilon_n}+y_n\right)-w_p(y_n)\right|
    \geq \gamma >0.
\end{displaymath}
Hence, defining $\displaystyle v_n(x) := \varepsilon_n \left(w_p\left(\frac{
      x}{\varepsilon_n}+y_n\right)-w_p(y_n)\right)$, we have
\begin{equation}
  \label{eq:90}
  v_n(0) = 0, \quad |v_n(\overline x_n)| \geq \gamma>0, \quad |\overline x_n| = 1.
\end{equation}
Moreover, Assumption~\ref{H5} implies
\begin{equation}
  \label{eq:4}
  \nabla v_n = \nabla w_p\left(\frac{\cdot}{\varepsilon_n}+y_n\right) \mathop{\longrightharpoonup}_{n\to +\infty} 0
  \quad \text{in}\quad L^2_{\rm loc}(\RR^d).
\end{equation}
Since $-\div\left(a\left(\frac x{\varepsilon_n} +y_n\right)\nabla \left(v_n(x)+p\cdot x\right)\right) = 0$,
Nash-Moser estimates  \cite[Theorem 8.24]{GT} imply that $v_n$ is bounded $C^{0,\beta}(B(0,2))$ for some
$\beta>0$. Hence, up to extracting a subsequence, it converges in $C^0(\overline B(0,1))$ to some $v\in C^0(\overline B(0,1))$. Now, extracting a
subsequence once again, we have $\overline x_n\to \overline x$, with $|\overline x|=1$. Hence, \eqref{eq:90} implies
\begin{displaymath}
  v(0) = 0, \quad |v(\overline x) |\geq \gamma>0, \quad |\overline x|=1.
\end{displaymath}
Since \eqref{eq:4} implies $\nabla v = 0$, we have reached a contradiction.

 Conversely, if (\ref{eq:ss-linearite1}) is satisfied, then there exists $A>0$ such that
 \begin{displaymath}
   \forall x\in B(0,A)^C, \quad \forall y\in \RR^d, \quad |w_p(x+y)-w_p(y)|\leq 1+|x|.
 \end{displaymath}
If necessary, we can take $A$ large enough to have $A\geq 2$.
In particular, we have $|w_p(x+y)- w_p(y) + p\cdot x|\leq 1+A+A|p|$ on $\partial B(0,A)$. Recalling that
\begin{displaymath}
  -\div_x\left[a(x+y)\nabla_x\left(w_p(x+y)-w_p(y)+p\cdot x\right)\right] = 0,
\end{displaymath}
in $B(0,A)$, this implies that 
\begin{displaymath}
  \forall x\in B(0,A), \quad |w_p(x+y)- w_p(y) + px|\leq 1+A+A|p|.
\end{displaymath}
Then, we apply the Caccioppoli inequality, which gives a constant $C$ depending only on the coefficient $a$ such that
\begin{multline*}
  \int_{B(y,1)} |\nabla w_p(z)+p|^2 dz \leq C \int_{B(y,2)} |w_p(z)-w_p(y)+ p\cdot(z-y)|^2dz \\= C\int_{B(0,2)}
  |w_p(x+y)-w_p(y)+p\cdot x|^2 dx \leq C \left(1+A+A|p|\right)\, |B(0,2)|.
\end{multline*}
This implies Assumption~\ref{H4}. In order to prove
  Assumption~\ref{H5}, we integrate by parts, finding
  \begin{multline*}
    \int_Q \nabla w_p \left(\frac x {\varepsilon_n} + y_n\right) dx = \int_{\partial Q} \varepsilon_n
    w_p\left(\frac x {\varepsilon_n} + y_n\right) n(x)dx \\= \sum_{j=1}^d \int_{\partial Q_j^+} \varepsilon_n
    \left[w_p\left(\frac x {\varepsilon_n} + y_n\right) - w_p\left(-\frac x {\varepsilon_n} + y_n\right)\right]n(x)dx.
  \end{multline*}
Here, $\partial Q_j^\pm$ denotes the faces of the cube $Q$, namely the set of equations $\{|x_k|<1/2 \ k\neq j, \
x_j=\pm 1/2\},$ and $n(x)$ is the outer normal to $Q$ at point $x$. Applying (\ref{eq:ss-linearite1}), we find \ref{H5}. 
\end{proof}

\paragraph{Logical links between our assumptions.} We have the following logical links between the assumptions
\begin{lemme}
  \label{lm:rq2}
  Assume that the matrix-valued coefficient $a$ satisfies Assumptions~\ref{H1} and \ref{H3}.
  \begin{enumerate}
  \item\label{item:1} If it satisfies Assumption~\ref{H7}, then it satisfies Assumptions~\ref{H4} and \ref{H5}.
  \item\label{item:2} If it satisfies Assumption~\ref{H8}, then it satisfies
    Assumption \ref{H6}.
  \end{enumerate}
\end{lemme}
\begin{proof}
  We first prove Assertion \ref{item:1}: if \ref{H7} holds, then clearly \eqref{eq:ss-linearite1} is
  satisfied. Hence, applying Lemma~\ref{lm:rq-sous-linearite}, we have \ref{H4} and \ref{H5}.

 As for Assertion \ref{item:2}, $B$ satisfies
  (\ref{eq:B-equation}), hence
  \begin{multline}\label{eq:91}
    \int_Q \left[\sum_{j=1}^d a_{ij}\left(\frac x {\varepsilon_n} + y_n\right) \left(\delta_{jk}+\partial_j w_{e_k}\left(\frac x
        {\varepsilon_n} + y_n \right)\right) - a_{jk}^*\right]dx = \int_Q \sum_{j=1}^d\partial_j
  B_k^{ij}\left(\frac x {\varepsilon_n} + y_n\right) dx \\ 
  =\sum_{j=1}^d\int_{\partial Q} \varepsilon_n\left(
  B_k^{ij}\left(\frac x {\varepsilon_n} + y_n\right) - B_k^{ij}(y_n)\right) e_j\cdot n(x) dx, 
  \end{multline}
where $n(x)$ is the outer normal to $Q$ at point $x$. Applying Assumption~\ref{H8}, we have, for any $x\in \partial
Q$, 
\begin{displaymath}
  \varepsilon_n\left|B_k^{ij}\left(\frac x {\varepsilon_n} + y_n\right) - B_k^{ij}(y_n) \right| \leq C
  \varepsilon_n |x|^{1-\nu}\varepsilon_n^{\nu-1} \leq C|x|\varepsilon_n^\nu.
\end{displaymath}
Inserting this estimate into~\eqref{eq:91}, we prove Assumption~\ref{H6}.
\end{proof}

\begin{remarque}
  The above proof implies that, if $B$ satisfies \eqref{eq:ss-linearite1}, that is,
\begin{equation}
  \label{eq:ss-linearite2}
  \lim_{|x|\to +\infty} \left(\sup_{y\in\RR^d} \frac{|B(x+y)-B(y)|}{1+|x|}\right) = 0,
\end{equation}
then it satisfies Assumption~\ref{H6}. Indeed, \eqref{eq:ss-linearite2} is sufficient, with \eqref{eq:91}, to prove
\ref{H6}. 
\end{remarque}



\paragraph{Uniform H-convergence.}

First, we prove that under Assumptions \ref{H1} through \ref{H6}, we have a
uniform H-convergence property, in the sense of Definition~\ref{def:uniform-H-cv}:
\begin{proposition}\label{prop:uniform-H-cv}
  Assume that the matrix-valued coefficient $a$ satisfies Assumptions \ref{H1} through \ref{H6}. Then, for any sequence
  $\left(y_n\right)_{n\in\NN}$ of $\RR^d$ and any sequence $\left(\varepsilon_n\right)_{n\in\NN}$ of positive
  numbers such that $\varepsilon_n \to 0$, and any bounded domain $\Omega$, the coefficient $a\left(\frac x {\varepsilon_n} + y_n\right)$ H-converges to
  $a^*$ on $\Omega$, where $a^*$ is defined by Assumption~\ref{H6}. 
\end{proposition}

\begin{proof}
  This is a standard application of homogenization tools (div-curl lemma in particular, see \cite[Lemma 1.1]{JKO}),
  so we skip it. The only important point is that all the estimates, hence the convergences, are uniform with
  respect to $y_n$.
\end{proof}

The following example proves that~\ref{H6} is not satisfied in general: in dimension $1$, define
\begin{displaymath}
  a(x) =
  \begin{cases}
    2 & \text{ if}\quad  2^n\leq x\leq 2^n +\frac{2^n}{\log(1+|n|)}, \quad n\in \ZZ, \\
    1 & \text{ otherwise.}
  \end{cases}
\end{displaymath}
Then it is clear that $a^* = 1$, and that the corrector is equal to $w' = \frac{1-a}a.$ Hence, using $y_n = 2^n$
and $\varepsilon_n = \log(1+|n|)2^{-n}$, we have
\begin{displaymath}
  \int_0^1 a\left(y_n + \frac x {\varepsilon_n} \right) w'\left(y_n + \frac x {\varepsilon_n} \right)dx =
 -1. 
\end{displaymath}
Hence, Assumption~\ref{H6} is not satisfied.

\paragraph{The matrix-valued coefficients $a$ and $a^T$.}

If the matrix-valued coefficient $a$ is not symmetric, we will in the sequel need to assume that both $a$ and $a^T$ satisfy assumptions \ref{H3}
through \ref{H8} (note that \ref{H1} and \ref{H2} are stable under transposition of $a$). 

In full generality,
the existence of strictly sublinear  correctors satisfying Assumptions~\ref{H4} and \ref{H5} for the coefficient
$a$ does not imply the existence of correctors for the adjoint coefficient $a^T$ satisfying the same properties, as
the following two-dimensional counter-example shows it. Note that it extends mutatis mutandis to any dimension
$d\geq 3$.

Consider
\begin{displaymath}
  a(x_1,x_2) =
  \begin{pmatrix}
    1 & \gamma(x_2) \\ 0 & 1
  \end{pmatrix},
\end{displaymath}
where $\gamma\in L^\infty(\RR)$, and $|\gamma|\leq 1,$ so that $a$ is indeed uniformly elliptic. Then $\div(a e_1) = \div(a e_2) = 0$, hence the correctors associated with $a$ are
all equal to $0$. We also compute
\begin{displaymath}
  \div\left(a^T e_1\right) = \gamma'(x_2). 
\end{displaymath}
Assume that $a^T$ admits a corrector for the vector $e_1$, and that it sastisfies \ref{H4} and \ref{H5}. We denote it by $w_{e_1}^T$. It is solution to
\begin{displaymath}
  \partial_1^2 w_{e_1}^T + \partial_2 \left(\gamma(x_2)\left(\partial_1 w_{e_1}^T+1\right)\right) + \partial_2^2
  w_{e_1}^T = 0.
\end{displaymath}
Hence, $v = \partial_1 w_{e_1}^T$ is a solution to $\Delta v + \partial_2\left(\gamma(x_2)\partial_1 v\right) =
0$. This is an elliptic equation, and $v\in L^2_{\rm unif}(\RR^d)$ according to \ref{H4}. Hence, applying the
Liouville theorem, $v$ is a constant. If this constant is not $0$, then $w_{e_1}^T$ cannot be sublinear at
infinity. Hence $v = 0$, which means that $w_{e_1}^T$ depends only on $x_2$. Hence $\partial_2^2 w_{e_1}^T =- \gamma'(x_2)$. This implies
\begin{displaymath}
  w_{e_1}^T(x_1,x_2) = C_2 - \int_0^{x_2} \left(\gamma(z)+C_1\right)dz.
\end{displaymath}
We choose for $\gamma$ the function
\begin{displaymath}
  \gamma = \chi * \gamma_0, \quad \text{with} \quad \gamma_0(z) = \frac12 \sum_{n\in\NN} \left(\mathbf{1}_{[2^{2n+1},2^{2n+2}]}(z)+\mathbf{1}_{[-2^{2n+2},-2^{2n+1}]}(z)\right),
\end{displaymath}
where $\chi$ is a smooth compactly-supported function such that $0\leq \chi\leq 1$ and $\int\chi =1$. For this
$\gamma$, it is easily seen that $w_{e_1}^T$ cannot be strictly sublinear at infinity. 

\paragraph{On the value of $\nu_r$.}

Let us point out that the value \eqref{eq:nu_r} of $\nu_r$ is optimal in the following sense: first, in the
periodic case, we recover the results of \cite{KLSGreenNeumann} (with $\nu_r=1$, that is, both the correctors and
the potential are bounded). Second, we have the following
example, in dimension one, in which $\left|\left(R^\varepsilon\right)'\right|$ is bounded from below, up to a
logarithmic term, by $\varepsilon^{\nu_r}$. It is unclear to us whether a similar example can, or not, be
constructed in higher dimensions. It however strongly suggests that the convergence rate stated is sharp.

Consider
\begin{displaymath}
  \widetilde a\in L^r(\RR), \quad 0\leq \widetilde a\leq 1, \quad \widetilde a(x) = \widetilde a(-x), \quad \text{and}\quad a^{per}
  = 1. 
\end{displaymath}
Then $a^*=1$, and the corrector is easily seen to be equal to
\begin{displaymath}
  w(x) = \widetilde w(x) = -\int_0^x \frac{\widetilde a(z)}{1+\widetilde a(z)}dz.
\end{displaymath}
In the special case $f=1$, if we solve \eqref{eq:equation} and \eqref{eq:homog_periodique} with $\Omega = ]-1,1[$, one easily computes
\begin{displaymath}
  \left(u^\varepsilon\right)'(x) = -\frac x {1+\widetilde a\left(\frac x \varepsilon\right)}, \quad \left(u^*\right)'(x)
  = -x. 
\end{displaymath}
Hence, computing $\left(R^\varepsilon\right)'$, we have
\begin{multline*}
  \left(R^\varepsilon\right)'(x) = \left(u^\varepsilon\right)'(x) - \left(u^*\right)'(x) - \varepsilon w\left(\frac
    x \varepsilon\right) \left(u^*\right)''(x) - w'\left(\frac x \varepsilon\right) \left(u^*\right)'(x)
\\= -\frac x {1+\widetilde a\left(\frac x \varepsilon\right)} + x -
  \varepsilon\int_0^{x/\varepsilon}\frac{\widetilde a(z)}{1+\widetilde a(z)}dz - x \frac{\widetilde a\left(\frac x
      \varepsilon\right)}{1+\widetilde a\left(\frac x \varepsilon\right) } = - \varepsilon  \int_0^{x/\varepsilon}\frac{\widetilde a(z)}{1+\widetilde a(z)}dz.
\end{multline*}
Hence, since $0\leq \widetilde a \leq 1$,
\begin{displaymath}
  \left|\left(R^\varepsilon\right)'(x)\right| \geq \frac \varepsilon 2 \int_0^{x/\varepsilon} \widetilde a(z)dz. 
\end{displaymath}
Using
\begin{displaymath}
  \widetilde a(z) = \frac 1 {\left(1+|z|\right)^{1/r} \left(1+\log(1+|z|)^{1+\delta}\right)^{1/r}}, \quad \delta >0,
\end{displaymath}
we find that, if $x>0$,
\begin{multline*}
  \left|\left(R^\varepsilon\right)'(x)\right| \geq \frac\varepsilon 2 \int_0^{x/\varepsilon}
  \frac{dz}{\left(1+|z|\right)^{1/r} \left(1+\log(1+|z|)^{1+\delta}\right)^{1/r}} \\ \geq \frac\varepsilon 2 \int_0^{x/\varepsilon}
  \frac{dz}{\left(1+|x|/\varepsilon\right)^{1/r} \left(1+\log(1+|x|/\varepsilon)^{1+\delta}\right)^{1/r}} \\= \frac\varepsilon 2 \frac{|x|}\varepsilon
  \frac{1}{\left(1+|x|/\varepsilon\right)^{1/r} \left(1+\log(1+|x|/\varepsilon)^{1+\delta}\right)^{1/r}}
\geq C\varepsilon^{1/r} \log\left(\varepsilon^{-1}\right)^{-(1+\delta)/r}.
\end{multline*}
Hence, estimate \eqref{eq:cv_H1} is optimal, up to logarithmic terms.


\subsection{Equation satisfied by the remainder}

We now prove
\begin{proposition}\label{pr:equation-reste}
  Assume~\ref{H1}, \ref{H3}, \ref{H4}, and that there exists $B\in W^{1,1}_{\rm loc}\left(\RR^d\right)$ solution to
  (\ref{eq:B-antisymetrique})-(\ref{eq:B-equation}). Then $R^\varepsilon$ defined by (\ref{eq:reste}) solves
\begin{equation}
  \label{eq:equation-reste}
  -\div\left(a\left(\frac x \varepsilon\right) \nabla R^\varepsilon\right) = \div\left(H^\varepsilon\right),
\end{equation}
\begin{equation}
  \label{eq:H-epsilon}
  H^\varepsilon_i(x) = \varepsilon \sum_{j,k=1}^d a_{ij}\left(\frac{x}{\varepsilon}\right) w_{e_k}\left( \frac{x}{\varepsilon}\right) \partial_j\partial_ku^*(x)
    -\varepsilon \sum_{j,k=1}^d  B^{ij}_k\left(\frac{x}{\varepsilon}\right)\partial_j\partial_k u^*(x),
\end{equation}
where $w_{e_k}$ is the corrector defined by (\ref{eq:correcteur}) with $p=e_k$, $1\leq k \leq d$.
\end{proposition}
\begin{proof} By definition of $R^\varepsilon$, that is, \eqref{eq:reste}, 
\begin{displaymath}
  \nabla R^\varepsilon(x) = \nabla u^\varepsilon(x) - \nabla u^*(x) - \sum_{j=1}^d \partial_j u^*(x)\nabla w_{e_j}\left(\frac x
    \varepsilon\right)  - \varepsilon \sum_{j=1}^d w_{e_j}\left(\frac x \varepsilon\right)
  \nabla \partial_j u^*(x).
\end{displaymath}
We have, using $\displaystyle -\div\left[a\left(\frac x \varepsilon\right)\nabla u^\varepsilon\right] = -\div\left(a^*\nabla u^*\right)$,
\begin{multline*}
  -\div\left(a\left(\frac x \varepsilon\right)\nabla R^\varepsilon\right) = \div\left[\left(a\left(\frac x
        \varepsilon\right) - a^*\right)\nabla u^* + a\left(\frac x \varepsilon\right)\sum_{j=1}^d \partial_j u^*(x)\nabla w_{e_j}\left(\frac x
    \varepsilon\right)  \right. \\ \left.+ \varepsilon\, a\left(\frac x \varepsilon\right) \sum_{j=1}^d w_{e_j}\left(\frac x \varepsilon\right)
  \nabla \partial_j u^*(x) \right],
\end{multline*}
in the sense of distributions. 
Using the definition \eqref{eq:Mik} of $M_k = (M_k^1,\dots ,M_k^d)^T$, this reads as
\begin{displaymath}
  -\div\left(a\left(\frac x \varepsilon\right)\nabla R^\varepsilon\right) =\div\left[- \sum_{k=1}^d \partial_k
    u^*(x)M_k\left(\frac x \varepsilon\right) + \varepsilon\, a\left(\frac x \varepsilon\right) \sum_{j=1}^d w_{e_j}\left(\frac x \varepsilon\right)
  \nabla \partial_j u^*(x) \right]
\end{displaymath}
We concentrate on the first term of the right-hand side, and use $\div(M_k)=0$:
\begin{displaymath}
  \div\left[\sum_{k=1}^d \partial_k u^*(x) M_k\left(\frac x \varepsilon\right)  \right] = \sum_{j=1}^d
  \sum_{k=1}^d  M_k^j\left(\frac x \varepsilon\right) \partial_j\partial_k u^*(x).
\end{displaymath}
We now use the potential $B$ defined by (\ref{eq:B-antisymetrique})-(\ref{eq:B-equation}), and write
\begin{multline*}
  \div\left[\sum_{k=1}^d \partial_k u^*(x) M_k\left(\frac x \varepsilon\right)  \right] = \sum_{i=1}^d \sum_{j=1}^d
  \sum_{k=1}^d \partial_i B^{ij}_k\left(\frac x \varepsilon\right) \partial_j\partial_k u^*(x) \\=
  \sum_{i=1}^d \partial_i\left(\sum_{j=1}^d\sum_{k=1}^d\varepsilon B^{ij}_k\left(\frac x
      \varepsilon\right)\partial_j\partial_k u^*(x) \right) - \sum_{i=1}^d \sum_{j=1}^d
  \sum_{k=1}^d B^{ij}_k\left(\frac x \varepsilon\right) \partial_i \partial_j\partial_k u^*(x).
\end{multline*}
The right-most term vanishes because, for each $k$, $B_k$ is skew-symmetric and $D^2 \left(\partial_k u^*\right)$ is symmetric.
\end{proof} 

Considering \eqref{eq:equation-reste}-\eqref{eq:H-epsilon}, a natural strategy to prove bounds on $R^\varepsilon$
is the following: first prove bounds on $H^\varepsilon$, then prove elliptic regularity estimates for the operator
$-\div(a(x/\varepsilon)\nabla\cdot)$ that are uniform with respect to $\varepsilon$.

The following two Lemmas achieve the first step of this strategy, establishing bounds on $H^\varepsilon$.

\begin{lemme}\label{lm:pr2.3}
  Assume \ref{H1} through \ref{H4}. Then, the correctors defined by
  Assumption~\ref{H3} satisfy
  \begin{equation}
    \label{eq:regularite_holder_w}
    \forall p\in\RR^d,\quad \nabla w_p \in C^{0,\alpha}_{\rm unif}(\RR^d).
  \end{equation}
If in addition Assumption~\ref{H8} holds, the potential $B$ defined by \eqref{eq:B-equation-2} satisfies
\begin{equation}
  \label{eq:regularite_hoder_B}
  \nabla B  \in C^{0,\alpha}_{\rm unif}(\RR^d).
\end{equation}
\end{lemme}
\begin{proof}
  Estimate (\ref{eq:regularite_holder_w}) is a direct consequence of elliptic estimates \cite[Theorem
  8.32]{GT}. Similarly, \eqref{eq:B-equation-2} reads $-\Delta B^{ij}_k = \partial_j M^i_k
  - \partial_i M^j_k = \div(M^i_k e_i - M^j_k e_j),$ where $M^j_k$ is defined by~(\ref{eq:Mik}). Using \eqref{eq:regularite_holder_w}, $M^j_k\in
  C^{0,\alpha}_{\rm unif}(\RR^d)$. Thus, applying \cite[Theorem
  8.32]{GT} again, we have~(\ref{eq:regularite_hoder_B}).
\end{proof}
\begin{lemme}\label{lm2.7}
  Assume \ref{H1}-\ref{H2}-\ref{H3}, and \ref{H7}-\ref{H8} for some $\nu>0$, and let $H^\varepsilon$ be
  defined by (\ref{eq:H-epsilon}). Then, for any $R>0$ and any $q\in[1,+\infty]$, if $D^2u^*\in
  L^q(\Omega)$, we have
  \begin{equation}
    \label{eq:estimation-H-L^p}
    \left\|H^\varepsilon\right\|_{L^q(\Omega\cap B(0,R))} \leq C \varepsilon^\nu R^{1-\nu}\left\| D^2
      u^*\right\|_{L^q(\Omega\cap B(0,R))},
  \end{equation}
where the constant $C$ does not depend on $D^2 u^*,R, \varepsilon$.

Moreover, $\alpha$ being defined by Assumption~\ref{H2}, for any $\beta\in[0,\alpha]$, if $u^*\in
C^{2,\beta}(\Omega)$, we have
\begin{equation}
  \label{eq:estimation-H-Holder}
  \left[ H^\varepsilon \right]_{C^{0,\beta}(\Omega\cap B(0,R))} \leq C \varepsilon^\nu R^{1-\nu}\left[D^2
    u^*\right]_{C^{0,\beta}(\Omega\cap B(0,R))} + C \varepsilon^{\nu-\beta} R^{1-\nu}\left\|D^2 u^*\right\|_{L^\infty(\Omega\cap B(0,R))},
\end{equation}
where $C$ does not depend on $D^2 u^*,R, \varepsilon$.
\end{lemme}
We recall here that the H\"older semi-norm $[\cdot]_{C^{0,\beta}(\Omega')}$ is defined by
\begin{equation}\label{eq:holder_semi_norm}
   \left[ v \right]_{C^{0,\beta}(\Omega')} = \sup_{x\neq y\in \Omega'} \frac{|v(x)-v(y)|}{|x-y|^\beta}.
\end{equation}

\begin{remarque}
  Lemma~\ref{lm2.7} is proved under Assumptions \ref{H1}, \ref{H2}, \ref{H3}, \ref{H7}, \ref{H8} only. However, applying
  Lemma~\ref{lm:rq2}, this in fact implies that Assumptions \ref{H4}, \ref{H5}, \ref{H6} are satisfied.  
\end{remarque}

\begin{proof}
  First, it is clear that
  \begin{multline}\label{eq:5}
    \left\|H^\varepsilon\right\|_{L^q(\Omega\cap B(0,R))} \leq \varepsilon\left(\|a\|_{L^\infty\left(\RR^d\right)}
      \sum_{k=1}^d\left\|w_{e_k} \left(\frac \cdot \varepsilon\right)\right\|_{L^\infty\left(\Omega\cap
          B(0,R)\right)} + \left\|B \left(\frac \cdot \varepsilon\right)\right\|_{L^\infty\left(\Omega\cap
          B(0,R)\right)}\right)\\ \times \left\| D^2 u^*\right\|_{L^q(\Omega\cap B(0,R))}.
  \end{multline}
Note that, $w_{e_k}$ and $B$ being defined up to the addition of a constant, we can always assume that $w_{e_k}(0) = 0$
and $B(0) = 0$. Hence, if $|x|/\varepsilon>1$, Assumptions~\ref{H7} and \ref{H8} imply
\begin{equation}
  \label{eq:92}
  \left|w_{e_k}\left(\frac x \varepsilon\right)\right|\leq C\left(\frac{|x|}\varepsilon\right)^{1-\nu}, \qquad
  \left| B \left(\frac x \varepsilon\right)\right|\leq C\left(\frac{|x|}\varepsilon\right)^{1-\nu}.
\end{equation}
If $|x|/\varepsilon\leq 1$, we use Lemma~\ref{lm:pr2.3}, which implies that $\nabla w_{e_k}\in L^\infty(\RR^d)$ and
$\nabla B\in L^\infty(\RR^d)$, whence
\begin{equation}
  \label{eq:93}
  \left|w_{e_k}\left(\frac x \varepsilon\right)\right|\leq C\frac{|x|}\varepsilon\leq C\left(\frac{|x|}\varepsilon\right)^{1-\nu}, \qquad
  \left| B \left(\frac x \varepsilon\right)\right|\leq C\frac{|x|}\varepsilon\leq C\left(\frac{|x|}\varepsilon\right)^{1-\nu}.
\end{equation}
Inserting \eqref{eq:92}-\eqref{eq:93} into \eqref{eq:5}, we find (\ref{eq:estimation-H-L^p}).

Next, we prove (\ref{eq:estimation-H-Holder}), writing
\begin{multline}\label{eq:7}
  \left[ H^\varepsilon \right]_{C^{0,\beta}(\Omega\cap B(0,R))} \leq \varepsilon \|a\|_{L^\infty\left(\RR^d\right)}
      \sum_{k=1}^d\left\|w_{e_k} \left(\frac \cdot
          \varepsilon\right)\right\|_{L^\infty\left(\Omega\cap B(0,R)\right)}\left[D^2 u^*\right]_{C^{0,\beta}(\Omega\cap B(0,R))} \\+
      \varepsilon  \|a\|_{L^\infty\left(\RR^d\right)}\sum_{k=1}^d\left[w_{e_k} \left(\frac \cdot
          \varepsilon\right)\right]_{C^{0,\beta}\left(\Omega\cap B(0,R)\right)}\left\|D^2 u^*\right\|_{L^\infty(\Omega\cap B(0,R))}
\\+ \varepsilon \left[a\left(\frac \cdot \varepsilon\right)\right]_{C^{0,\beta}(\Omega\cap B(0,R))}
      \sum_{k=1}^d\left\|w_{e_k} \left(\frac \cdot \varepsilon\right)\right\|_{L^\infty\left(\Omega\cap B(0,R)\right)}
      \left\|D^2 u^*\right\|_{L^\infty(\Omega\cap B(0,R))} \\
 +\varepsilon \left\|B \left(\frac \cdot \varepsilon\right)\right\|_{L^\infty\left(\Omega\cap B(0,R)\right)}\left[D^2
   u^*\right]_{C^{0,\beta}(\Omega\cap B(0,R))} 
+\varepsilon \left[B \left(\frac \cdot \varepsilon\right)\right]_{C^{0,\beta}(\Omega\cap B(0,R))}\left\|D^2 u^*\right\|_{L^\infty\left(\Omega\cap B(0,R)\right)} 
\end{multline}
Here again, we use \eqref{eq:92}-\eqref{eq:93}, which imply
\begin{equation}
  \label{eq:6}
   \sum_{k=1}^d\left\|w_{e_k} \left(\frac \cdot
          \varepsilon\right)\right\|_{L^\infty\left(\Omega\cap B(0,R)\right)} + \left\|B \left(\frac \cdot
          \varepsilon\right)\right\|_{L^\infty\left(\Omega\cap B(0,R)\right)} \leq C \varepsilon^{\nu-1} R^{1-\nu}.
\end{equation}
Using Assumption~\ref{H2}, we also have, since $\beta\leq \alpha$,
\begin{equation}
  \label{eq:8}
\left[a\left(\frac \cdot \varepsilon\right)\right]_{C^{0,\beta}(\Omega\cap B(0,R))} \leq C\varepsilon^{-\beta}.
\end{equation}
Using (\ref{eq:regularite_holder_w}), we have, for $|x-y|<\varepsilon$,
\begin{multline*}
  \frac{\left|\varepsilon w_{e_k}\left(\frac x \varepsilon\right) -\varepsilon w_{e_k}\left(\frac y
        \varepsilon\right) \right|}{|x-y|^\beta} \leq C \left\|\nabla w_{e_k}\right\|_{L^\infty(\RR^d)}
  |x-y|^{1-\beta} \\= C \left\|\nabla w_{e_k}\right\|_{L^\infty(\RR^d)}
  |x-y|^{\nu-\beta}|x-y|^{1-\nu} \leq C\left\|\nabla w_{e_k}\right\|_{L^\infty(\RR^d)} \varepsilon^{\nu-\beta}R^{1-\nu}.
\end{multline*}
If $|x-y|>\varepsilon$, we use Assumption~\ref{H7}, which implies
\begin{displaymath}
  \frac{\left|\varepsilon w_{e_k}\left(\frac x \varepsilon\right) -\varepsilon w_{e_k}\left(\frac y
        \varepsilon\right) \right|}{|x-y|^\beta} \leq C \varepsilon^\nu |x-y|^{1-\beta-\nu} = C \varepsilon^\nu
  |x-y|^{-\beta} |x-y|^{1-\nu} \leq C \varepsilon^{\nu-\beta}R^{1-\nu}.
\end{displaymath}
Collecting the above estimates, we obtain
\begin{equation}
  \label{eq:9}
  \left[\varepsilon w_{e_k}\left(\frac x \varepsilon\right) \right]_{C^{0,\beta}(\Omega\cap B(0,R))} \leq C
  \varepsilon^{\nu-\beta} R^{1-\nu}.
\end{equation}
A similar argument allows to prove that 
\begin{equation}
  \label{eq:10}
  \left[\varepsilon B\left(\frac x \varepsilon\right) \right]_{C^{0,\beta}(\Omega\cap B(0,R))} \leq C \varepsilon^{\nu-\beta} R^{1-\nu}.
\end{equation}
Hence, inserting (\ref{eq:6}), (\ref{eq:8}), (\ref{eq:9}), (\ref{eq:10}) into (\ref{eq:7}), we find (\ref{eq:estimation-H-Holder}).
\end{proof}

Next, we are going to prove elliptic regularity estimates for
the operator $-\div(a(x/\varepsilon)\nabla \cdot)$ that are uniform in $\varepsilon$. This will in turn allow to
prove estimates on $R^\varepsilon$ using \eqref{eq:equation-reste}.

\section{Estimates in the homogeneous case}
\label{sec:homogeneous-case}

Our aim is now to prove, as a first step, that, if the coefficient $a$ satisfies \ref{H1} through \ref{H6}, then a solution $v^\varepsilon$ to 
\begin{equation}\label{eq:96}
  -\div \left(a\left(\frac x \varepsilon +y\right) \nabla v^\varepsilon\right) = 0,
\end{equation}
satisfies Lipschitz bounds \emph{uniformly} in $\varepsilon>0$ and $y\in \RR^d$. To this end, we apply the compactness method of
Avellaneda and Lin \cite{AvellanedaLin}. Loosely speaking, since as $\varepsilon$ vanishes, the equation
homogenizes into
\begin{displaymath}
  \div \left(a^* \nabla v^*\right) = 0,
\end{displaymath}
for which Lipschitz bounds hold, thus, for $\varepsilon$ sufficiently small, such bounds should survive. On the other
hand, for $\varepsilon$ "large", bounded away from zero, they also hold, uniformly, by standard elliptic regularity
results, thus, intuitively, the result.

\subsection{H\"older estimates}
The main result of this Section is a generalization of \cite[Lemma 24]{AvellanedaLin} to the present
setting:
\begin{theoreme}\label{th1}
  Assume that the matrix-valued coefficient $a$ satisfies \ref{H1} through
  \ref{H6}. Assume that $\Omega$ is a $C^{2,1}$ bounded domain, that $\beta\in ]0,1[$, $y\in\RR^d,$ and $g\in
  C^{0,\beta}\left(\overline B(0,1)\right).$ Assume that $v^\varepsilon$ is a solution to
  \begin{displaymath}
    \left\{
      \begin{aligned}
        &-\div\left(a\left(\frac x \varepsilon+y\right) \nabla v^\varepsilon\right) = 0 & \text{ in } \Omega\cap
        B(0,1), \\
        &v^\varepsilon = g & \text{ in } \left(\partial\Omega\right) \cap \overline B(0,1).
      \end{aligned}
\right.
  \end{displaymath}
Then there exists a constant depending only on $a$, $\beta$ and $\Omega$ such that
\begin{equation}
  \label{eq:estimation-holder-homogene}
  \|v^\varepsilon\|_{C^{0,\beta}(\Omega\cap B(0,1/2))} \leq C\left( \|g\|_{C^{0,\beta}(\left(\partial\Omega\right)\cap \overline B(0,1))}
    + \|v^\varepsilon\|_{L^2(\Omega\cap
        B(0,1))}\right)
\end{equation}
\end{theoreme}

In order to prove Theorem~\ref{th1}, we first assume that $B(0,1)\subset \Omega$. In such a case,
\eqref{eq:estimation-holder-homogene} becomes an interior estimate. Its proof is the matter of Lemma~\ref{lm7} and
Lemma~\ref{lm8} below. In a second step, we allow for $B(0,1)$ to intersect $\partial \Omega$ and prove the same type
of estimates (Lemma~\ref{lm22} and \ref{lm23} below).

\bigskip

We first prove a result that generalizes \cite[Lemma 7]{AvellanedaLin} (with $f=0$ there)
to the present setting. 
\begin{lemme}\label{lm7}
    Assume \ref{H1} through
  \ref{H6}, and let $\beta\in ]0,1[$. There exists $\theta \in
  ]0,1/4[$ depending only on $\mu$ (see Assumption~\ref{H1}) and $\beta$, there exists a $\varepsilon_0>0$
  depending only on $a$, $\beta$ and $\theta$, such that, $\forall y\in \RR^d$, if $v^\varepsilon$ is a solution to 
  \begin{equation}
    \label{eq:14}
        -\div\left(a\left(\frac x \varepsilon+y\right) \nabla v^\varepsilon\right) = 0  \text{ in } B(0,1),
  \end{equation}
then
\begin{equation}
  \label{eq:initialisation}
  \fint_{B(0,\theta)} \left|v^\varepsilon -\fint_{B(0,\theta)}v^\varepsilon \right|^2 dx \leq \theta^{2\beta}
  \fint_{B(0,1)} \left|v^\varepsilon(x)\right|^2 dx.
\end{equation}
\end{lemme}
\begin{proof} We reproduce the proof of \cite[Lemma 7]{AvellanedaLin}, and
  use, instead of periodicity, uniform H-convergence.
  Consider $v^*$ a solution to $-\div(a^*\nabla v^*) = 0$ in $B(0,1/2)$. The matrix $a^*$ being constant, we have
  \begin{multline*}
    \fint_{B(0,\theta)} \left|v^* -\fint_{B(0,\theta)}v^* \right|^2 dx \leq \theta^2 \left\|\nabla
      v^*\right\|_{L^\infty(B(0,\theta))}^2 \leq
\theta^{2} \left\|\nabla v^*\right\|_{L^\infty(B(0,1/4))}^2 \\ \leq
  C\theta^2\fint_{B(0,1/2)} \left|v^*(x)\right|^2 dx.
  \end{multline*}
The right-most inequality is a consequence of elliptic regularity results. It may be proved by successively
applying \cite[Theorem 8.32]{GT}, and \cite[Theorem 8.24]{GT}. Hence, for $\theta$ sufficiently small,
  \begin{equation}\label{eq:94}
    \fint_{B(0,\theta)} \left|v^* -\fint_{B(0,\theta)}v^* \right|^2 dx \leq \frac{\theta^{2\beta}}{2^{d+1}}
    \fint_{B(0,1/2)} \left|v^*(x)\right|^2 dx 
  \end{equation}
We then fix such a $\theta$ and argue by contradiction to prove that $v^\varepsilon$ satisfies
\begin{equation}
  \label{eq:initialisation_bis}
  \fint_{B(0,\theta)} \left|v^\varepsilon -\fint_{B(0,\theta)}v^\varepsilon \right|^2 dx \leq \frac{\theta^{2\beta}}{2^d}
  \fint_{B(0,1/2)} \left|v^\varepsilon(x)\right|^2 dx.
\end{equation}
If it does not hold, then we can build sequences
$\varepsilon_n\to 0$ and $y_n\in \RR^d$ such that
\begin{equation}
  \label{eq:non_initialisation}
  \fint_{B(0,\theta)} \left|v^{\varepsilon_n} -\fint_{B(0,\theta)}v^{\varepsilon_n} \right|^2 dx > \frac{\theta^{2\beta}}{2^{d}}
  \fint_{B(0,1/2)} \left|v^{\varepsilon_n}(x)\right|^2 dx,
\end{equation}
where $v^{\varepsilon_n}$ solves \eqref{eq:14} (with $\varepsilon = \varepsilon_n$ and $y=y_n$). Normalizing $v^{\varepsilon_n}$ if necessary, we may assume that
$\displaystyle \fint_{B(0,1)} \left|v^{\varepsilon_n}\right|^2 = 1.$
Applying the Caccioppoli inequality \cite[page 76]{Giaquinta}, the sequence $\left(v^{\varepsilon_n}\right)_{n\in\NN}$ is bounded in
$H^1(B(0,1/2))$. Hence we can extract a subsequence converging strongly in $L^2(B(0,1/2))$ and weakly in
$H^1(B(0,1/2))$, to some limit $v^*\in H^1(B(0,1/2))$. Applying Proposition~\ref{prop:uniform-H-cv} (this where we
use assumptions \ref{H1} through \ref{H6}), we see that $v^*$ is a solution to
$-\div(a^*\nabla v^*) = 0$ in $B(0,1/2)$. Hence it satisfies \eqref{eq:94}. On the other hand, strong convergence
in $L^2(B(0,1/2)$ allows to pass to the limit in \eqref{eq:non_initialisation}, reaching a contradiction. We have
proved \eqref{eq:initialisation_bis}, which clearly implies \eqref{eq:initialisation}.
\end{proof}

Exactly as in \cite[Lemma 8]{AvellanedaLin} (with $f=0$ there), a proof by induction (which we therefore do not include here) from
Lemma~\ref{lm7} allows to prove the following
\begin{lemme}
  \label{lm8}
  Under the assumptions of Lemma~\ref{lm7}, let $\theta\in ]0,1/4[$ and $\varepsilon_0$ be given by Lemma~\ref{lm7}.
  If $\varepsilon\in ]0,\theta^k\varepsilon_0[,$ and if $v^\varepsilon$ satisfies (\ref{eq:14}), then 
\begin{displaymath}
  \fint_{B(0,\theta^k)} \left|v^\varepsilon -\fint_{B(0,\theta^k)}v^\varepsilon \right|^2 dx \leq \theta^{2k\beta}
  \fint_{B(0,1)} \left|v^\varepsilon(x)\right|^2 dx.
\end{displaymath}
\end{lemme}

Following the sketch of the proof of \cite[Lemma 10]{AvellanedaLin} (with $f=0$ and $g=0$ there), and using uniform H-convergence
where periodicity was used in \cite{AvellanedaLin}, we obtain 

\begin{lemme}  \label{lm22}
    Assume~\ref{H1} through
  \ref{H6} with $\beta\in ]0,1[$, and that $\Omega$ is a $C^{1,\alpha}$ bounded domain such that, say,
  $0\in\partial\Omega$. There exists $\theta \in ]0,1/4[$ and $\varepsilon_0>0$ depending only on $a$, $\beta$ and
  $\Omega$, such that, for any $\varepsilon<\varepsilon_0$, any $y\in\RR^d$, and any solution $v^\varepsilon$ of
  \begin{equation}
    \label{eq:16}\left\{
      \begin{aligned}
        &-\div\left(a\left(\frac x \varepsilon+y\right) \nabla v^\varepsilon\right) = 0 & \text{ in } \Omega\cap
        B(0,1), \\
        &v^\varepsilon = 0 & \text{ on } \left(\partial\Omega\right) \cap \overline B(0,1),
      \end{aligned}
\right.
  \end{equation}
we have
\begin{equation}
  \label{eq:17}
  \fint_{\Omega\cap B(0,\theta)} \left|v^\varepsilon(x)\right|^2 dx \leq\theta^{2\beta} \fint_{\Omega\cap B(0,1)}
  \left|v^\varepsilon(x)\right|^2 dx.
\end{equation}
\end{lemme}
\begin{proof} 
  Assume temporarily that $v^*$ is a solution to 
  \begin{equation}
    \label{eq:18}\left\{
      \begin{aligned}
        &-\div\left(a^* \nabla v^*\right) = 0 & \text{ in } \Omega\cap
        B(0,1), \\
        &v^* = 0 & \text{ in } \left(\partial\Omega\right) \cap \overline B(0,1).
      \end{aligned}
\right.
  \end{equation}
In particular, we have $v^*(0) = 0$, hence, for any $\theta \in ]0,1/4[$,
\begin{displaymath}
  \fint_{\Omega\cap B(0,\theta)} \left|v^*\right|^2 \leq C\theta^2 \left\|\nabla v^*\right\|^2_{L^\infty(\Omega\cap B(0,1/4))}.
\end{displaymath}
Applying the boundary gradient estimate \cite[Corollary 8.36]{GT}, we have
\begin{displaymath}
\left\|\nabla v^*\right\|_{L^\infty(\Omega\cap B(0,1/4))}\leq C
\|v^*\|_{L^\infty(\Omega\cap B(0,1/2))},
\end{displaymath}
hence
\begin{equation}\label{eq:19}
  \fint_{\Omega\cap B(0,\theta)} \left|v^*\right|^2 \leq C\theta^2 \left\|v^*\right\|^2_{L^\infty(\Omega\cap B(0,1/2))}
\end{equation}
We apply \cite[Theorem 8.25]{GT}. This gives $\displaystyle
\left\|v^*\right\|_{L^\infty(\Omega\cap B(0,1/2))} \leq C \fint_{\Omega\cap B(0,1)} \left|v^*\right|^2$. Hence,
inserting this estimate into \eqref{eq:19}, we find
\begin{displaymath}
   \fint_{\Omega\cap B(0,\theta)} \left|v^*\right|^2 \leq C \theta^2 \fint_{\Omega\cap B(0,1)} \left|v^*\right|^2
\end{displaymath}
Thus, for $\theta>0$ sufficiently small,
\begin{equation}
  \label{eq:20}
  \fint_{\Omega\cap B(0,\theta)} \left|v^*\right|^2 \leq \frac{\theta^{2\beta}}2 \fint_{\Omega \cap B(0,1)} \left|v^*\right|^2.
\end{equation}
We now fix $\theta>0$ to this value, and argue by contradiction: if \eqref{eq:17} does not hold, then one can find
a sequence $\varepsilon_n\to 0$ and a sequence $y_n$ such that,
for each $n$ the solution $v^{\varepsilon_n}$ of \eqref{eq:16} (with $\varepsilon=\varepsilon_n$, $y = y_n$) satisfies 
\begin{equation}\label{eq:95}
  \fint_{\Omega\cap B(0,\theta)} |v^{\varepsilon_n}|^2 > \theta^{2\beta} \fint_{\Omega\cap B(0,1)} |v^{\varepsilon_n}|^2.
\end{equation}
Multiplying $v^{\varepsilon_n}$ by a normalizing constant if necessary, we may assume that
\begin{equation}
  \label{eq:21}
  \fint_{\Omega\cap B(0,1)} \left|v^{\varepsilon_n}\right|^2 = 1.
\end{equation}
The sequence $(v^{\varepsilon_n})_{n\in\NN}$ is bounded in $H^1(\Omega \cap B(0,1/2))$ according to
Caccioppoli's inequality \cite[Proposition 2.1, p 76]{Giaquinta}. Hence, we can extract weak convergence in
$L^2(\Omega\cap B(0,1)\cap H^1(\Omega\cap B(0,1/2))$ and strong convergence in $L^2(\Omega\cap B(0,1/2))$. We Denote by $v^*$ its
limit. Inequality \eqref{eq:95} implies
\begin{displaymath}
 \theta^{2\beta} \fint_{\Omega\cap B(0,1)} |v^*|^2 \leq \theta^{2\beta}\liminf_{n\to+\infty} \fint_{\Omega\cap B(0,1)} |v^{\varepsilon_n}|^2 \leq
  \liminf_{n\to+\infty} \fint_{\Omega\cap B(0,\theta)} |v^{\varepsilon_n}|^2 = \fint_{\Omega\cap B(0,\theta)} |v^*|^2.
\end{displaymath}
In addition, Proposition~\ref{prop:uniform-H-cv}, which is valid since we assumed \ref{H1} through \ref{H6}, allows to prove
that $v^*$ is a solution to \eqref{eq:18}, hence satisfies \eqref{eq:20}. We therefore reach a contradiction,
concluding the proof. 
\end{proof}


Here again, using an induction argument as in the proof of \cite[Lemma 11]{AvellanedaLin} (with $f=0$ there), we have

\begin{lemme}
  \label{lm23}
  Under the same assumptions as those of Lemma~\ref{lm22}, with $\theta>0$ and $\varepsilon_0>0$ defined by the
  conclusion of Lemma~\ref{lm22}, we have, for any integer $k\geq 0$, if $\varepsilon<\theta^k \varepsilon_0$,
  \begin{displaymath}
    \fint_{\Omega\cap B(0,\theta^k)} |v^\varepsilon|^2 \leq \theta^{2k\beta} \fint_{\Omega\cap B(0,1)} |v^\varepsilon|^2.
  \end{displaymath}
\end{lemme}

\bigskip

The four above Lemmas allow us to proceed with the proof of Theorem~\ref{th1}. We first deal with the case of
interior estimate, that is, $\partial \Omega\cap \overline B(0,1)= \emptyset$, then we prove the general case.

\begin{proof}[Proof of Theorem~\ref{th1}]
  Assume first that $\partial \Omega\cap \overline B(0,1)= \emptyset$. Then the proof is exactly that of
  \cite[Lemma~9]{AvellanedaLin} with $f=0$, in which periodicity is not used. Next, if $\partial \Omega\cap \overline B(0,1)\neq \emptyset$, we
  follow the proof of \cite[Lemma 24]{AvellanedaLin}.
\end{proof}

\subsection{Lipschitz estimates}

In this Section, we prove the following result, which is the generalization of \cite[Lemma 16]{AvellanedaLin} (with
$f=0$ there) to he present setting:
\begin{theoreme}\label{lem16}
  Assume \ref{H1} through \ref{H6}. Let $y\in\RR^d$, $R>0$, and assume that $v^\varepsilon\in H^1(B(0,2R))$ is a solution to
  \begin{displaymath}
    -\div\left(a\left(\frac x \varepsilon + y \right) \nabla v^\varepsilon (x)\right) = 0 \quad \text{in} \quad B(0,2R).
  \end{displaymath}
Then, there exists a constant $C$ depending only on the coefficient $a$ such that
\begin{equation}
  \label{eq:estimation_lispchitz}
  \sup_{x\in B(0,R)} \left|\nabla v^\varepsilon(x)\right|\leq \frac C R \left(\fint_{B(0,2R)} \left|v^\varepsilon\right|^2\right)^{1/2}.
\end{equation}
\end{theoreme}


As we did for the proof of Hölder estimates above, we are going to apply the proof of \cite{AvellanedaLin},
replacing, when necessary, periodicity by assumptions \ref{H3} through \ref{H6}.

We first prove a result that is the generalization of \cite[Lemma 14]{AvellanedaLin} (with $f=0$ there) to our setting.
\begin{lemme}
  \label{lm14}
  Assume that the matrix-valued coefficient $a$ satisfies Assumptions~\ref{H1} through \ref{H6}, and
  let $\gamma\in ]0,1[.$ Then there exists $\varepsilon_0>0$ and $\theta \in ]0,1/4[$ depending only on $a$ and
  $\gamma$ such that, if $\varepsilon < \varepsilon_0$ and if $v^\varepsilon$ satisfies 
  \begin{equation}\label{eq:30}
    -\div\left(a\left(\frac x \varepsilon + y \right) \nabla v^\varepsilon (x)\right) = 0 \quad \text{in} \quad B(0,1),
  \end{equation}
then
\begin{multline}
  \label{eq:31}
  \sup_{x\in B(0,\theta)} \left|v^\varepsilon(x) - v^\varepsilon(0) - \sum_{j=1}^d \left[x_j + \varepsilon
      \left(w_{e_j}\left(\frac x \varepsilon+ y\right) - w_{e_j}(y)\right)
    \right]\fint_{B(0,\theta)}\partial_j v^\varepsilon \right| \\ \leq \theta^{1+\gamma}\left(\fint_{B(0,1)} |v^\varepsilon|^2\right)^{1/2}
\end{multline}
\end{lemme}
\begin{proof}
  As in the proof of Lemma~\ref{lm22}, we argue by contradiction. Let $v^*\in H^1(B(0,1/2))$ be a solution to
\begin{equation}
  \label{eq:32}
  -\div\left(a^* \nabla v^*\right) = 0 \quad \text{in} \quad B\left(0,\frac12\right).
\end{equation}
Since $a^*$ is constant, $\partial_j\partial_j v^*$ is also a solution to
(\ref{eq:32}). Hence, applying the interior H\"older estimate of \cite[Theorem 8.24]{GT}, we have
$\|D^2 v^*\|_{C^{0,\beta}(B(0,1/8))} \leq C \|D^2 v^*\|_{L^2(B(0,1/4))},$
where $C$ and $\beta$ depend only on $a^*$. Hence,
\begin{equation}
  \label{eq:33}
  \|D^2 v^*\|_{L^\infty(B(0,1/8))} \leq C \|D^2 v^*\|_{L^2(B(0,1/4))}.
\end{equation}
Then applying the Caccioppoli inequality \cite[Proposition 2.1, p 76]{Giaquinta} twice, we infer
\begin{equation}
  \label{eq:34}
  \int_{B(0,1/4)} |D^2 v^*|^2 \leq C \int_{B(0,1/2)} |v^*|^2.
\end{equation}
In (\ref{eq:33}) and (\ref{eq:34}), the constant $C$ depends only on $a^*$.
Using a Taylor expansion, and applying (\ref{eq:33}) and (\ref{eq:34}) to bound the remainder, we find that
there exists a constant $C_0$ depending only on $a^*$ such that
\begin{displaymath}
  \sup_{x\in B(0,\theta)} \left|v^*(x) - v^*(0) - x\cdot\fint_{B(0,\theta)} \nabla v^*\right|\leq C_0\theta^2
  \left(\fint_{B(0,1/2)} |v^*|^2 \right)^{1/2}.
\end{displaymath}
Hence, choosing $\theta$ such that $C_0\theta^2  \leq \frac{\theta^{1+\gamma}}{2^{1+d/2}}$, we find that $v^*$ satisfies
(\ref{eq:31}) $w_{e_j}$ is replaced by $0$, that is,
\begin{equation}
  \label{eq:36}
  \sup_{x\in B(0,\theta)} \left|v^* - v^*(0) - x\cdot \fint_{B(0,\theta)}\nabla v^* \right| \\ \leq
  \frac{\theta^{1+\gamma}}2\frac 1 {2^{d/2}}\left(\fint_{B(0,1/2)} |v^*|^2\right)^{1/2}.
\end{equation}
The condition on $\theta$ reads $\theta \leq \left(2^{1+d/2}C_0 \right)^{-1/(1-\gamma)}$,
which depends only on $a^*$, $d$ and $\gamma$.

\medskip

Next, we assume
that (\ref{eq:31}) does not hold, that is, there exists sequences $\varepsilon_n\to 0$,
$y_n\in\RR^d$ and $v^{\varepsilon_n} \in H^1(B(0,1))$ such that (\ref{eq:30}) holds (with
$\varepsilon=\varepsilon_n,$ $y=y_n$, $v^\varepsilon=v^{\varepsilon_n}$), and 
\begin{multline}
  \label{eq:35}
  \sup_{x\in B(0,\theta)} \left|v^{\varepsilon_n}(x) - v^{\varepsilon_n}(0) - \sum_{j=1}^d \left[x_j + \varepsilon_n
      \left(w_{e_j}\left(\frac x {\varepsilon_n}+ y\right) - w_{e_j}(y)\right)
    \right]\fint_{B(0,\theta)}\partial_j v^{\varepsilon_n} \right| \\ > \theta^{1+\gamma}\left(\fint_{B(0,1)} |v^{\varepsilon_n}|^2\right)^{1/2}.
\end{multline}
Multiplying $v^{\varepsilon_n}$ by some constant if necessary, we may assume that $\displaystyle \fint_{B(0,1)} |v^{\varepsilon_n}|^2 =
1$. Applying the Caccioppoli inequality, we deduce that $v^{\varepsilon_n}$ is bounded in $H^1(B(0,1/2))$, hence, up to
extracting a subsequence, we have $v^{\varepsilon_n}\longrightharpoonup
v^*$ in $H^1(B(0,1/2))\cap L^2(B(0,1))$. Applying Proposition~\ref{prop:uniform-H-cv}
(thereby using Assumptions \ref{H1} through \ref{H6}), we prove that $v^*$
satisfies (\ref{eq:32}), hence (\ref{eq:36}). Next, applying Theorem~\ref{th1}, we have
$\|v^{\varepsilon_n}\|_{C^{0,\beta}(B(0,1/2))} \leq C$. This allows to
pass to the limit in the first two terms of the left-hand side of
(\ref{eq:35}). Weak convergence in $H^1(B(0,1/2))$ allows to pass to the limit in the term
$\displaystyle\fint_{B(0,\theta)} \partial_j v^{\varepsilon_n}$. Moreover,
Assumptions \ref{H1} through \ref{H5} allow to apply
Lemma~\ref{lm:rq-sous-linearite}, which implies that, for all $j\in \{1,2,\dots,d\}$,
\begin{displaymath}
  \sup_{y\in\RR^d} \sup_{x\in B(0,1)} \varepsilon \left| w_{e_j}\left(\frac x \varepsilon + y \right) -
    w_{e_j}(y)\right| \mathop{\longrightarrow}_{\varepsilon\to 0} 0.
\end{displaymath}
Hence, passing to the limit in (\ref{eq:35}), we find
\begin{multline*}
  \sup_{x\in B(0,\theta)} \left|v^*(x) - v^*(0) - x\cdot \fint_{B(0,\theta)} \nabla v^* \right| \geq
  \liminf_{n\to+\infty}\, \theta^{1+\gamma} \left(\fint_{B(0,1)}\left|v^{\varepsilon_n}\right|^2\right)^{1/2}\\ \geq
  \theta^{1+\gamma} \left(\fint_{B(0,1)} \left|v^*\right|^2\right)^{1/2} = \frac{\theta^{1+\gamma}}{2^{d/2}} \left(\fint_{B(0,1/2)} \left|v^*\right|^2\right)^{1/2},
\end{multline*}
and we reach a contradiction with (\ref{eq:36}).
\end{proof}

As in \cite[Lemma 15]{AvellanedaLin} (with $f=0$ there), an induction argument allows to prove the following
\begin{lemme}
  \label{lm15}
  Assume \ref{H1} through \ref{H6}, and that $\gamma
  \in ]0,1[$. Let $\theta$ and $\varepsilon_0$ be given by Lemma~\ref{lm14}. There exists $C>0$ depending only on
  $\theta$ such that, for any $y\in \RR^d$, if $0<\varepsilon\leq \varepsilon_0\theta^n$, $n\in \NN$, and if
  $v^\varepsilon\in H^1(B(0,1))$ satisfies (\ref{eq:30}), we have
  \begin{displaymath}
    \sup_{x\in B(0,\theta^{n+1})} \left| v^\varepsilon(x) - v^\varepsilon(0) - \sum_{j=1}^d \left[ x_j +
        \varepsilon\left(w_{e_j}\left(\frac x \varepsilon+y\right) - w_{e_j}(y) \right)\right] \kappa_j(n)\right|
    \leq \theta^{(1+n)(1+\gamma)} \left\|v^\varepsilon\right\|_{L^\infty(B(0,1))},
  \end{displaymath}
where $\kappa_j(n)$ satisfies
\begin{equation}
  \label{eq:38}
  \sup_{1\leq j\leq d} |\kappa_j(n)|\leq C \|v^\varepsilon\|_{L^\infty(B(0,1))}\sum_{\ell=0}^n \theta^{\gamma \ell}.
\end{equation}
\end{lemme}
\begin{remarque}
  In (\ref{eq:38}), the important point is that $C$ depends on $\theta$ but not on $n$. Hence, since $\theta<1$, it
  implies $\displaystyle \sup_{1\leq j\leq d} |\kappa_j(n)|\leq C \|v^\varepsilon\|_{L^\infty(B(0,1))}$, and will
  be used as such in the sequel. However, the form (\ref{eq:38}) is more convenient for the induction proof.
\end{remarque}

\begin{proof}[Proof of Theorem~\ref{lem16}] This exactly the proof of
  \cite[Lemma 16]{AvellanedaLin}, based on Lemma~\ref{lm14} and
  Lemma~\ref{lm15}. We therefore omit it.
\end{proof}

\section{Estimates in the inhomogeneous case}
\label{sec:inhomogeneous-case}

In this Section, we deal with the non-homogeneous case, that is, the
case when the right-hand side of (\ref{eq:96}) is some $\div(f)$, $f\in L^2(\Omega)$, with $f\neq 0$.

We first prove estimates on the Green function $G^\varepsilon$ of the operator $-\div\left(a\left(\frac
    x \varepsilon\right)\nabla \cdot\right)$ with homogeneous Dirichlet boundary conditions. This uses the results
on the homogeneous case, since $x\mapsto G^\varepsilon(x,y)$ and $x\mapsto \nabla_y G^\varepsilon(x,y)$ are
solution to $-\div\left(a\left(\frac x \varepsilon\right) \nabla_x v\right) = 0$ in any open set that does not
contain $y$. Then, we use the representation $u^\varepsilon(x) = \int_\Omega G^\varepsilon(x,y)f(y)dy$
to prove estimates in the case $f\neq 0$.

\subsection{Green function estimates}

First, we recall that in \cite{GruterWidman}, $G^\varepsilon$ was proved
 to exist and be unique in $W^{1,1}_0(\Omega)$. In addition, the following estimates were established in \cite{GruterWidman,Dolzmann}:
\begin{equation}
  \label{eq:green_0}
  \forall x\neq y\in\Omega, \quad 0\leq G^\varepsilon(x,y)\leq \frac C {|x-y|^{d-2}},
\end{equation}
\begin{displaymath}
  \|\nabla_x G^\varepsilon(\cdot, y)\|_{L^{\frac d {d-1},\infty}(\Omega)} +   \|\nabla_y G^\varepsilon(x,\cdot)\|_{L^{\frac d {d-1},\infty}(\Omega)} \leq C,
\end{displaymath}
where $C$ depends only on $\|a\|_{L^\infty}$ and on its ellipticity constant. Here, $L^{p,\infty}$ denotes the
Marcinkiewicz space of order $p$, as defined, e.g., in \cite{BerghLofstrom}.

We now show
\begin{theoreme}
  \label{th:estimations_G}
  Let $d\geq 3$. Assume \ref{H1} through
  \ref{H6}. Let $\Omega$ be a $C^{1,\alpha}$ bounded domain. Denote by $G^\varepsilon$ the Green function of the operator
  $-\div\left(a\left(\frac x \varepsilon\right)\nabla \cdot\right)$ on $\Omega$ with homogeneous Dirichlet boundary
  conditions. For any $\Omega_1\subset\subset\Omega$, we have the following estimates:
  \begin{enumerate}
  \item \label{item:3}
    \begin{equation}
      \label{eq:green_1}
      \forall x\in \Omega_1,\ \forall y \in \Omega, \ x\neq y, \quad |\nabla_x G^\varepsilon(x,y)|\leq \frac C {|x-y|^{d-1}}.
    \end{equation}
\item \label{item:5}If in addition $a^T$ satisfies Assumptions~\ref{H3}, \ref{H4}, \ref{H5} and \ref{H6}, then we have
  \begin{equation}
    \label{eq:green_2}
      \forall y\in \Omega_1,\ \forall x \in \Omega, \ x\neq y,  \quad |\nabla_y G^\varepsilon(x,y)|\leq \frac C {|x-y|^{d-1}},
  \end{equation}
  \begin{equation}
    \label{eq:green_3}
      \forall x\in \Omega_1,\ \forall y \in \Omega_1, \ x\neq y,  \quad |\nabla_y\nabla_x G^\varepsilon(x,y)|\leq \frac C {|x-y|^{d}}.
  \end{equation}
  \end{enumerate}
In (\ref{eq:green_1})-(\ref{eq:green_2})-(\ref{eq:green_3}), the various
constants $C$ depend only on the coefficient $a$, on $\Omega$ and on
$\Omega_1$.
\end{theoreme}

The above result is actually contained in \cite{AvellanedaLin}, if the coefficient $a$ is assumed to be
periodic. However, it is not stated as such, and its proof, which may be
found in the course of the proof of \cite[Lemma
17]{AvellanedaLin}, is different from the one we present here. 


\begin{proof}
  We first prove Assertion~\ref{item:3}. We define $\delta = \inf\left\{|x-y|,\ x\in\Omega_1,\ y\in\partial\Omega\right\}>0.$
Let $x_0\in \Omega_1$, $y_0\in \Omega\setminus\{x_0\}$. We set
  \begin{displaymath}
    R = \frac12 \min\Bigl(d(x_0,\partial\Omega)\ ,\ |x_0-y_0|\Bigr)
  \end{displaymath}
We have
\begin{equation}
  \label{eq:49}
  2R\leq |x_0-y_0|\leq CR,
\end{equation}
where the constant $C$ is $C = 2$  if $R=\frac12|x_0-y_0|$, and
$C=\frac{2\diam(\Omega)}\delta$ otherwise. In particular it depends only on $\Omega$ and
$\Omega_1$. Since $y_0\notin B(x_0,R),$
\begin{equation}\label{eq:52}
  -\div_x\left(a\left(\frac x \varepsilon\right) \nabla_x G^\varepsilon(x,y_0)\right) = 0 \quad \text{in}\quad B(x_0,R).
\end{equation}
Applying Theorem~\ref{lem16} to $x\mapsto G^\varepsilon(x,y_0)$, we have
\begin{displaymath}
  \left|\nabla_x G^\varepsilon(x_0,y_0)\right|\leq \frac C R \left(\fint_{B(x_0,R/2)} \left|G^\varepsilon(x,y_0)\right|^2dx\right)^{1/2}.
\end{displaymath}
Using (\ref{eq:green_0}), (\ref{eq:49}), and the triangle inequality,
$|x-y_0|\geq |x_0-y_0| - |x-x_0|$, we have
\begin{displaymath}
  \left(\fint_{B(x_0,R/2)}\left|G^\varepsilon(x,y_0)\right|^2 dx\right)^{1/2} \leq C \left(\fint_{B(x_0,R/2)} \frac 1
    {R^{2(d-2)}} dx \right)^{1/2} = \frac C {R^{d-2}}.
\end{displaymath}
Hence,
\begin{displaymath}
  |\nabla_x G^\varepsilon(x_0,y_0)|\leq \frac C {R^{d-1}}.
\end{displaymath}
Using (\ref{eq:49}) again, we find (\ref{eq:green_1}).

Next, we prove Assertion \ref{item:5}. It is well-known (see \cite[Theorem 1.3]{GruterWidman}) that the Green function
$G_T^\varepsilon$ of the operator $-\div\left(a^T\left(\frac x \varepsilon\right)\nabla \cdot\right)$ with
homogeneous Dirichlet condition satisfies $G_T^\varepsilon(x,y) = G^\varepsilon(y,x)$. Since $a^T$ satisfies
Assumptions~\ref{H1}, \ref{H2}, \ref{H3}, \ref{H4}, \ref{H5}, \ref{H6}, $G_T^\varepsilon$ satisfies
(\ref{eq:green_1}). This clearly implies (\ref{eq:green_2}).

Finally, we note that $\nabla_y G(x,y_0)$ is also a solution to (\ref{eq:52}). Hence, applying the proof of
Assertion~\ref{item:5} to $\nabla_y G$, we find (\ref{eq:green_3}).
\end{proof}

\subsection{$W^{1,p}$ estimates}

We now prove $W^{1,p}$ estimates on the solution $v^\varepsilon$ of (\ref{eq:second_membre_div}) below. The
following Proposition is the generalization of \cite[Theorem 2.4.1]{Shen} to the present setting. 

\begin{proposition}
  \label{pr:1}
  Assume \ref{H1} through \ref{H6}. Let
  $q\in ]2,+\infty[$, $y\in \RR^d,$ $R>0$ and $H\in L^q(B(0,2R),\RR^d)$. Assume that $v^\varepsilon\in H^1(B(0,2R))$ is a solution
  to
  \begin{equation}
    \label{eq:second_membre_div}
    -\div\left(a\left(\frac x\varepsilon+y\right) \nabla v^\varepsilon\right) = \div(H)\quad \text{in}\quad B(0,2R).
  \end{equation}
Then, there exists $C>0$ depending only on the coefficient $a$ and on $q$ (in particular it does not depend on $y$ nor on
$\varepsilon$) such that
\begin{displaymath}
  \left(\fint_{B(0,R)} |\nabla v^\varepsilon|^q \right)^{1/q} \leq C \left(\fint_{B(0,2R)} |H|^q
  \right)^{1/q} + C\left(\fint_{B(0,2R)} |\nabla v^\varepsilon|^2 \right)^{1/2}.
\end{displaymath}
\end{proposition}

Before we get to the proof of Proposition~\ref{pr:1}, we first state the following Lemma, which is a simple
consequence of \cite[Theorem 2.4]{ShenCalderon} (see also \cite[Theorem 2.3.1]{Shen}):
\begin{lemme}
  \label{lmA16}
  Let $B_0 = B(x_0,R_0)$ be a ball of $\RR^d$, and $F\in L^2(4B_0)$. Let $2<q_1<q_2$, $f\in
  L^{q_1}(4B_0)$. Assume that there exists $K>0$ such that for any ball $B\subset 2B_0$ with $2|B|\leq
  |B_0|$, there exists $F_1\in L^2(2B)$ and $F_2\in L^{q_2}(2B) $ such that
  \begin{align*}
    &|F|\leq |F_1|+|F_2| \quad\text{in} \quad 2B, \\
    &\left(\fint_{2B} |F_1|^2 \right)^{1/2} \leq K \sup_{B \subset B' \subset 4B_0}
      \left(\fint_{B'} |f|^2 \right)^{1/2}, \\
&    \left(\fint_{2B} |F_2|^{q_2} \right)^{1/{q_2}} \leq K\left[\left(\fint_{4B} |F|^2 \right)^{1/2}+ \sup_{B \subset B' \subset 4B_0}
      \left(\fint_{B'} |f|^2 \right)^{1/2}\right],
  \end{align*}
where the supremum is taken over any ball $B'$ such that $B \subset B' \subset B(x_0,4R_0)$. Then, $F\in
L^{q_1}(B_0)$, and
\begin{displaymath}
  \left(\fint_{B_0} |F|^{q_1} \right)^{1/{q_1}} \leq C\left[\left(\fint_{4B_0} |F|^2 \right)^{1/2}+ 
      \left(\fint_{4B_0} |f|^{q_1} \right)^{1/q_1}\right],
\end{displaymath}
where $C$ depends on $K,q_1,q_2$ only. 
\end{lemme}

\begin{proof}[Proof of Proposition~\ref{pr:1}] The proof follows the lines of \cite[Theorem 2.4.1]{Shen}. However,
  since the setting is slightly different, we reproduce it here for the sake of clarity and for the reader's convenience.

  Let $x_0\in B(0,2R)$ and $R_0>0$ such that $B_0:=B(x_0,R_0)$ satisfies $8B_0 \subset B(0,2R)$. We intend to apply
  Lemma~\ref{lmA16} to $F=\nabla v^\varepsilon$ and $f=H$. For this purpose, we fix $y_0\in 2B_0$ and $R_1>0$ such
  that $B:= B(y_0,R_1)\subset 2B_0.$
  \begin{displaymath}
    v^\varepsilon = v_1^\varepsilon + v_2^\varepsilon,
  \end{displaymath}
where $v^\varepsilon_1$ satisfies
\begin{displaymath}
\left\{
      \begin{aligned}
        &-\div\left(a\left(\frac x {\varepsilon}+y\right) \nabla v_1^\varepsilon\right) = \div(H) & \text{ in } B(y_0,4R_1), \\
        &v_1^\varepsilon = 0 & \text{ in } \partial\left(B(y_0,4R_1)\right).
      \end{aligned}
\right.  
\end{displaymath}
Multiplying this equation by $v_1^\varepsilon$ and integrating by parts, we have
\begin{equation}
  \label{eq:53}
  \left(\fint_{B(y_0,4R_1)} |\nabla v^\varepsilon_1|^2 \right)^{1/2} \leq C \left(\fint_{B(y_0,4R_1)} |H|^2\right)^{1/2},
\end{equation}
where $C$ depends only on the ellipticity constant of $a$. On the other hand, $v_2^\varepsilon$ satisfies
\begin{displaymath}
  -\div\left(a\left(\frac x {\varepsilon}+y\right) \nabla v_2^\varepsilon\right) = 0\quad \text{in }  B(y_0,4R_1).
\end{displaymath}
Thus, applying Theorem~\ref{lem16} to $\displaystyle v_2^\varepsilon - \fint_{B(y_0,4R_1)} v_2^\varepsilon$, we have
\begin{displaymath}
  \|\nabla v_2^\varepsilon\|_{L^\infty(B(y_0,2R_1))} \leq \frac C R_1 \left(\fint_{B(y_0,4R_1)} \left| v_2^\varepsilon -
      \fint_{B(y_0,4R_1)} v_2^\varepsilon\right|^2\right)^{1/2},
\end{displaymath}
where the constant $C$ depends only on the coefficient $a$. Applying the Poincar\'e-Wirtinger inequality, this implies
\begin{displaymath}
  \|\nabla v_2^\varepsilon\|_{L^\infty(B(y_0,2R_1))} \leq C' \left(\fint_{B(y_0,4R_1)} |\nabla v_2^\varepsilon|^2\right)^{1/2}.
\end{displaymath}
The constant $C'$ is equal to
$C' = \frac C {R_1} C_{PW}(B(y_0,4R_1)) = C C_{PW}(B(y_0,4)),$ due to the scaling of the constant $C_{PW}$ in the
Poincar\'e-Wirtinger inequality. Hence, $C'$ depends
only on $a$. On the other hand, using (\ref{eq:53}) and the triangle inequality, we have
\begin{multline*}
  \left(\fint_{B(y_0,4R_1)} |\nabla v_2^\varepsilon|^2\right)^{1/2} \leq \left(\fint_{B(y_0,4R_1)} |\nabla
    v_1^\varepsilon|^2\right)^{1/2}+\left(\fint_{B(y_0,4R_1)} |\nabla v^\varepsilon|^2\right)^{1/2} \\ \leq C
  \left(\fint_{B(y_0,4R_1)} |H|^2\right)^{1/2} + \left(\fint_{B(y_0,4R_1)} |\nabla v^\varepsilon|^2\right)^{1/2}.
\end{multline*}
Thus,
\begin{equation}
  \label{eq:54}
  \left\|\nabla v_2^\varepsilon\right\|_{L^\infty(B(y_0,2R_1))}\leq C  \left(\fint_{B(y_0,4R_1)} |H|^2\right)^{1/2} + \left(\fint_{B(y_0,4R_1)} |\nabla v^\varepsilon|^2\right)^{1/2}.
\end{equation}
Collecting (\ref{eq:53}) and (\ref{eq:54}), we may apply Lemma~\ref{lmA16} (with $B_0=B(x_0,R_0)$, $q_1=q$,
$q_2=2q_1$, $f=H$, $F_1 = \nabla v_1^\varepsilon$, $F_2 = \nabla v_2^\varepsilon$, and $B=B(y_0,R_1)$) finding
\begin{displaymath}
  \left(\fint_{B_0} |\nabla v^\varepsilon|^q\right)^{1/q} \leq C \left[\left(\fint_{4B_0} |H|^q\right)^{1/q} + \left(\fint_{4B_0} |\nabla v^\varepsilon|^2\right)^{1/2} \right].
\end{displaymath}
This is valid for any $x_0$ and $R_0>0$ such that $B(x_0,8R_0)\subset B(0,2R)$. Hence, covering $B(0,R)$ by a
finite number of such balls, we conclude the proof. 
\end{proof}

\subsection{Lipschitz estimates}

Note that Proposition~\ref{pr:1} does not include the case $q=+\infty$. However, using the estimates we have proved
on the gradient of $G^\varepsilon$ in Theorem~\ref{th:estimations_G}, we are able to now derive Lipschitz estimates:
\begin{proposition}
  \label{pr:2}
  Assume that the coefficients $a$ and $a^T$ satisfy Assumptions~\ref{H1} through \ref{H6}. Let
  $\beta>0$ and $R>\varepsilon>0$, and assume that $H\in C^{0,\beta}(B(0,2R))$. Then, there exists a constant $C>0$ depending
  only on $a$ and $\beta$ such that, if $v^\varepsilon$ satisfies \eqref{eq:second_membre_div}, then
  \begin{multline}
    \label{eq:estimation_lipschitz_non_homogene}
    \|\nabla v^\varepsilon\|_{L^\infty(B(0,R))}\leq C \left(\fint_{B(0,2R)} |\nabla v^\varepsilon|^2\right)^{1/2} +
    C \varepsilon^\beta [H]_{C^{0,\beta}(B(0,2R))} \\+ C \ln\left(1+ \frac R \varepsilon\right) \|H\|_{L^\infty(B(0,2R))}.
  \end{multline}
\end{proposition}
We recall here that $[\cdot]_{C^{0,\beta}(B(0,2R))}$ denotes the H\"older semi-norm on $B(0,2R)$ (see
\eqref{eq:holder_semi_norm}).

\medskip

Proposition~\ref{pr:2} is a generalization of \cite[Lemma 3.5]{KLSGreenNeumann}, in two ways. First, we replace,
here, the periodicity assumption by \ref{H1} through \ref{H6}. Second, in
\cite{KLSGreenNeumann}, Lemma 3.5 is stated only for the specific case where $v^\varepsilon= R^\varepsilon$ defined
by \eqref{eq:reste}, hence $H = H^\varepsilon$ defined by \eqref{eq:H-epsilon}. Due to these differences, we provide
below a complete proof, although the ideas are contained in \cite{KLSGreenNeumann}.

\begin{proof}
  We split the proof in several steps: first, introducing a cut-off function, we write $v^\varepsilon$ as an integral of $G^\varepsilon$, which is
  the Green function of the operator $-\div\left(a\left(\frac x \varepsilon+y\right)\nabla \cdot\right)$ with homogeneous
    Dirichlet boundary conditions on $B(0,2R)$. Then, we use this representation and Theorem~\ref{th:estimations_G} to prove \eqref{eq:estimation_lipschitz_non_homogene}.

\smallskip

\noindent\underline{Step 1: introduction of a cut-off function and use of the Green function.} We define $\phi\in
C^\infty_c(B(0,3R/2))$ such that
\begin{displaymath}
  0\leq\phi\leq 1, \quad \phi = 1 \text{ in } B(0,5R/4), \quad \|\nabla \phi\|_{L^\infty(B(0,2R))} \leq \frac C R, \quad \|D^2
  \phi\|_{L^\infty(B(0,2R))} \leq \frac C {R^2}.
\end{displaymath}
We clearly have $\|\nabla(\phi v^\varepsilon)\|_{L^\infty(B(0,R))} = \|\nabla
v^\varepsilon\|_{L^\infty(B(0,R))}.$ Moreover,
\begin{displaymath}
  -\div\left(a\left(\frac z \varepsilon + y\right) \nabla \left(\phi v^\varepsilon\right)\right) =
  -\div\left(v^\varepsilon a\left(\frac z \varepsilon + y\right) \nabla \phi \right) -  \left(a\left(\frac z \varepsilon
    + y\right)\nabla v^\varepsilon\right)\cdot \nabla \phi + \phi\div(H).
\end{displaymath}
Hence, multiplying by $G^\varepsilon(x,z)$ and integrating with respect to $z$ over $B(0,2R)$, 
\begin{eqnarray*}
  \phi(x)v^\varepsilon(x) &=& -\int_{B(0,2R)} G^\varepsilon(x,z)  \left[a\left(\frac z \varepsilon + y\right)\nabla
  v^\varepsilon(z)\right] \cdot \nabla\phi(z)dz \\ &&+ \int_{B(0,2R)} \nabla_z G^\varepsilon(x,z)\cdot
\left(v^\varepsilon(z)  a\left(\frac z \varepsilon + y\right)\nabla \phi(z)\right)dz  \\ &&- 
  \int_{B(0,2R)} \nabla_z\left(G^\varepsilon(x,z)\phi(z)\right)\cdot H(z)dz\\
&=:& v_1^\varepsilon(x) + v_2^\varepsilon(x) + v_3^\varepsilon(x).
\end{eqnarray*}

\medskip

\noindent\underline{Step 2: bound on $v_1^\varepsilon$.} Let $x\in B(0,R)$. Since $\nabla \phi$ vanishes in
$B(0,5R/4)$ and outside $B(0,3R/2)$, we have
\begin{displaymath}
  |\nabla v^\varepsilon_1(x)|\leq \int_{B(0,3R/2)\setminus B(0,5R/4)} |\nabla_x G^\varepsilon(x,z)|\, \left|a\left(\frac
      z \varepsilon + y\right) \right|\,
  |\nabla v^\varepsilon(z)|\, |\nabla \phi(z)| dz.
\end{displaymath}
Successively using $|\nabla \phi|\leq \frac C R$, estimate \eqref{eq:green_1}, and $B(0,3R/2)\setminus B(0,5R/4)
\subset B(0,2R)\setminus B(0,R),$ we deduce
\begin{equation}\label{eq:55}
  |\nabla v^\varepsilon_1(x)|
  \leq \frac C {R^d} \int_{B(0,2R)\setminus B(0,R)} |\nabla v^\varepsilon| \leq C\left(\fint_{B(0,2R)}|\nabla v^\varepsilon|^2\right)^{1/2}.
\end{equation}
\medskip

\noindent\underline{Step 3: bound on $v_2^\varepsilon$.} Similar arguments allow to prove that
\begin{multline}\label{eq:977}
  |\nabla v_2^\varepsilon(x)| \leq \int_{B(0,3R/2)\setminus B(0,5R/4)} |\nabla_x\nabla_zG^\varepsilon(x,z)|\,
  |v^\varepsilon(z)|\, \left|a\left(\frac z \varepsilon + y\right)\right|\, |\nabla \phi(z)|dz \\
  \leq \frac C R \left(\int_{B(0,3R/2)\setminus B(0,5R/4)}|\nabla_x\nabla_z G(x,z)|^2dz\right)^{1/2} \|v^\varepsilon\|_{L^2(B(0,2R))},
\end{multline}
the last inequality coming from the Cauchy-Schwarz inequality. We then apply \eqref{eq:green_3}, which implies
\begin{equation}\label{eq:58}
  \left(\int_{B(0,3R/2)\setminus B(0,5R/4)}|\nabla_x\nabla_z G(x,z)|^2dz\right)^{1/2} \leq \frac C {R^{d/2}}.
\end{equation}
We point out that adding a constant to $v^\varepsilon$ does not change \eqref{eq:second_membre_div}, hence
we may assume that $\int_{B(0,2R)} v^\varepsilon = 0$. So, using the Poincar\'e-Wirtinger inequality, we have
\begin{displaymath}
  \|v^\varepsilon\|_{L^2(B(0,2R))} \leq C\left(R^2\int_{B(0,2R)} |\nabla v^\varepsilon|^2\right)^{1/2} =C R^{1+d/2}
  \left(\fint _{B(0,2R)} |\nabla v^\varepsilon|^2\right)^{1/2},
\end{displaymath}
where $C$ does not depend on $R$. Inserting this inequality and \eqref{eq:58} into \eqref{eq:977}, we infer
\begin{equation}
  \label{eq:56}
  |\nabla v_2^\varepsilon(x)|\leq C \left(\fint_{B(0,2R)}|\nabla v^\varepsilon_2|^2\right)^{1/2}.
\end{equation}
\medskip
\noindent\underline{Step 4: bound on $v_3^\varepsilon$.} We fix here again $x\in B(0,R)$. Integrating by parts, we have
\begin{equation}\label{eq:57}
  \int_{B(0,2R)} \nabla_z \left(G^\varepsilon(x,z)\phi(z)\right)dz = 0, 
\end{equation}
hence
\begin{displaymath}
  v_3^\varepsilon(x) = \int_{B(0,2R)} \nabla_z\left(G^\varepsilon(x,z)\phi(z)\right)\cdot (H(z) - H(x))dz.
\end{displaymath}
We differentiate this equalilty with respect to $x$, and use~\eqref{eq:57} again, finding
\begin{displaymath}
  \nabla v_3^\varepsilon(x) = \int_{B(0,2R)} \nabla_z\left(\nabla_xG^\varepsilon(x,z)\phi(z)\right)\cdot(H(z) -
  H(x))dz - \underbrace{\int_{B(0,2R)} \nabla_z\left(G^\varepsilon(x,z)\phi(z)\right)\cdot \nabla_x
  H(x)dz}_{=0}.
\end{displaymath}
Thus,
\begin{multline*}
  |\nabla v_3^\varepsilon(x)|\leq  \int_{B(0,2R)} |\phi(z)|\, |\nabla_z\nabla_xG^\varepsilon(x,z)|\, |H(z) -
  H(x)|dz \\+ \int_{B(0,2R)} |\nabla \phi(z)|\, |\nabla_xG^\varepsilon(x,z)|\, |H(z) -
  H(x)|dz
\end{multline*}
Using that $\nabla \phi$ vanishes in
$B(0,5R/4)$ and outside $B(0,3R/2)$, that $|\nabla \phi|\leq C/R$, and \eqref{eq:green_1}, we have
\begin{multline*}
   \int_{B(0,2R)} |\nabla \phi(z)|\, |\nabla_xG^\varepsilon(x,z)|\, |H(z) -
  H(x)|dz \\ \leq \frac C R \|H\|_{L^\infty(B(0,2R))} \int_{B(0,3R/2)\setminus B(0,5R/4)} \frac 1 {|x-z|^{d-1}}dz \leq
  C \|H\|_{L^\infty(B(0,2R))}.
\end{multline*}
Moreover, using \eqref{eq:green_3} and the fact that $H$ is $\beta$-H\"older continuous, we also have,
\begin{multline*}
  \int_{B(0,2R)} |\phi(z)|\, |\nabla_z\nabla_xG^\varepsilon(x,z)|\, |H(z) - H(x)|dz \leq C [H]_{C^{0,\beta}(B(0,2R))} \int_{B(x,\varepsilon)} \frac{|x-z|^\beta}{|x-z|^d}dz \\
 + 2\|H\|_{L^\infty(B(0,2R))} \int_{B(0,2R)\setminus B(x,\varepsilon)}\frac{dz}{|x-z|^d}.  
\end{multline*}
The integral in the right-most term of the right-hand side is bounded as follows (we use here $|x|\leq R$):
\begin{displaymath}
   \int_{B(0,2R)\setminus B(x,\varepsilon)}\frac{dz}{|x-z|^d} \leq
  C\int_\varepsilon^{\max(3R,\varepsilon)} \frac{r^{d-1}dr}{r^d} =  C
  \ln\left(\frac{\max(3R,\varepsilon)}\varepsilon\right) \leq C\ln\left(1+\frac R \varepsilon\right).
\end{displaymath}
Hence,
\begin{equation}
  \label{eq:60}
  |\nabla v_3^\varepsilon(x)|\leq  C \varepsilon^\beta [H]_{C^{0,\beta}(B(0,2R))} + C\ln\left(1+R\varepsilon^{-1}\right)\|H\|_{L^\infty(B(0,2R))}.
\end{equation}
Collecting \eqref{eq:55}, \eqref{eq:56}, \eqref{eq:60}, we have proved \eqref{eq:estimation_lipschitz_non_homogene}.
\end{proof}

\begin{remarque}
  In Propoisition~\ref{pr:2}, we have assumed that both coefficients $a$ and $a^T$ satisfy Assumptions \ref{H1} through \ref{H6}. The result
  however still holds if only $a$ satisfies those assumptions. Indeed,  the assumption on $a^T$ is only used for the
  proof of \eqref{eq:58} and \eqref{eq:60}: in both cases, we have used the pointwise bound \eqref{eq:green_3} on $\nabla_x\nabla_yG$, but the only
  relevant bound for proceeding with the proof of Proposition~\ref{pr:2}  is an $L^2$ bound, which can alternately be obtained
  using \eqref{eq:green_1} and the Cacciopoli inequality (see \cite[Section 2.5.3]{theseMJ} for the details).
\end{remarque}

\subsection{Convergence rates for Green functions}

We now prove the following convergence result of $G^\varepsilon$ to the Green function $G^*$ of the operator
$-\div(a^*\nabla \cdot)$ with homogeneous Dirichlet conditions on $\Omega$. It is the extension, in our setting, of
\cite[Theorem 3.3]{KLSGreenNeumann}

\begin{theoreme}
  \label{th:cv_Green}
  Assume that the matrix-valued coefficients $a$ and $a^T$ satisfy Assumptions \ref{H1} through
  \ref{H6}, and \ref{H7}-\ref{H8} for some $\nu>0$. Let $\Omega$ be a domain of class $C^{2,1}$, and denote by $G^\varepsilon$ and
  $G^*$ the Green functions of the operators $-\div\left(a\left(\frac x \varepsilon\right)\nabla \cdot\right)$ and
  $-\div(a^*\nabla \cdot)$, respectively, with homogeneous Dirichlet boundary conditions on $\Omega$. Then there exists a
  constant $C>0$ depending only on $a$, $\Omega$ and $\nu$ such that
  \begin{equation}
    \label{eq:cv_Green}
    \forall x\neq y\in \Omega, \quad \left|G^\varepsilon(x,y) - G^*(x,y)\right|\leq C \frac{\varepsilon^\nu}{|x-y|^{d+\nu-2}}.
  \end{equation}
\end{theoreme}

The proof of Theorem~\ref{th:cv_Green} replicates that of \cite[Theorem 3.3]{KLSGreenNeumann}, but we need to
everywhere keep track of the use of Assumptions \ref{H7}-\ref{H8} and check that these properties are sufficient to
proceed at each step of the arguments. 

\medskip

We prove the following lemma, which is a generalization of \cite[Lemma 3.2]{KLSGreenNeumann}:
\begin{lemme}
  \label{KLSlem3.2}
Assume that the matrix-valued coefficient $a$ satisfies \ref{H1} through
  \ref{H6}, and \ref{H7}-\ref{H8} for some $\nu>0$. Let $\Omega$ be a $C^{2,1}$ bounded domain, $x_0\in \overline\Omega$,
  $R>0$, $q_1>d$ and $q_2\in ]1,+\infty[$. Assume that $u^\varepsilon\in H^1(\Omega\cap B(x_0,4R))$ and $u^*\in
  W^{2,q_1}(\Omega\cap B(x_0,4R))$ satisfy
  \begin{displaymath}
\left\{
      \begin{aligned}
        &-\div\left(a\left(\frac x {\varepsilon}\right) \nabla u^\varepsilon\right) = -\div(a^*\nabla u^*) & \text{
          in } \Omega\cap B(x_0,4R), \\
        &u^\varepsilon = u^* & \text{ on } \left(\partial\Omega\right)\cap \overline B(x_0,4R).
      \end{aligned}
\right.  
  \end{displaymath}
Then, there exists a constant $C$ depending only on $a$, $\Omega$, $q_1$ and $q_2$ such that
\begin{multline}
  \label{eq:cv_prelim}
  \left\|u^\varepsilon - u^*\right\|_{L^\infty(\Omega\cap B(x_0,R))} \leq CR^{-\frac d{q_2}}\|u^\varepsilon -
  u^*\|_{L^{q_2}(\Omega\cap B(x_0,4R))} + C\varepsilon^\nu R^{1-\nu} \|\nabla u^*\|_{L^\infty(\Omega\cap
    B(x_0,4R))} \\+ C \varepsilon^\nu R^{2-\frac d {q_1} - \nu} \|D^2 u^*\|_{L^{q_1}(\Omega\cap B(x_0,4R))}
\end{multline}
\end{lemme}
\begin{proof}
  We follow the proof of \cite[Lemma 3.2]{KLSGreenNeumann}, adapting it when necessary. First, since the problem is
  translation invariant, we may assume that $x_0=0$. Then, we define a smooth open set $\widetilde \Omega$ such that
  \begin{displaymath}
    \Omega\cap B(0,2R)\subset \widetilde\Omega\subset \Omega\cap B(0,4R).
  \end{displaymath}
We define the remainder $R^\varepsilon$ by \eqref{eq:reste}. We know that it satisfies \eqref{eq:equation-reste},
with $H^\varepsilon$ defined by \eqref{eq:H-epsilon}. Next, we split $R^\varepsilon$ into $R^\varepsilon =
R_1^\varepsilon + R_2^\varepsilon$, where $R_1^\varepsilon$ is defined as the unique solution of
  \begin{equation}\label{eq:72}
\left\{
      \begin{aligned}
        &-\div\left(a\left(\frac x {\varepsilon}\right) \nabla R_1^\varepsilon\right) = -\div(H^\varepsilon) & \text{
          in } \widetilde \Omega, \\
        &R_1^\varepsilon = 0 & \text{ on } \partial\widetilde\Omega.
      \end{aligned}
\right.  
  \end{equation}
Hence, $R_2^\varepsilon$ satisfies
  \begin{equation}\label{eq:73}
\left\{
      \begin{aligned}
        &-\div\left(a\left(\frac x {\varepsilon}\right) \nabla R_2^\varepsilon\right) = 0 & \text{
          in } \widetilde \Omega, \\
        &R_2^\varepsilon(x) = \varepsilon\sum_{j=1}^d w_{e_j}\left(\frac x \varepsilon\right)\partial_j u^*(x) &
        \text{ on } \partial\widetilde\Omega \cap \partial\Omega.
      \end{aligned}
\right.  
  \end{equation}
We use a scaling argument, defining $\overline R_2^\varepsilon(x) = R_2^\varepsilon(x/R)$, $\overline a(x) =
a(x/R)$, $\overline w_{e_j}(x) = w_{e_j}(x/R)$, and $\overline u^*(x) = u^*(x/R)$. Writing down the equation
satisfied  by $\overline R_2^\varepsilon$, we are thus in the case $R=1$ and we may apply De Giorgi-Nash estimate. Scaling
back to the original unknown $R_2^\varepsilon$, this implies 
\begin{displaymath}
  \left\|R_2^\varepsilon\right\|_{L^\infty(\Omega\cap B(0,R))} \leq C \left\|\varepsilon\sum_{j=1}^d
    w_{e_j}\left(\frac x \varepsilon\right)\partial_j u^*(x) \right\|_{L^\infty(\widetilde\Omega)} + \frac C
  {R^{d/q_2}} \|R_2^\varepsilon\|_{L^{q_2}(\widetilde\Omega)}.
\end{displaymath}
Using Assumption~\ref{H7} and the triangle inequality, this implies
\begin{equation}
  \label{eq:62}
  \left\|R_2^\varepsilon\right\|_{L^\infty(\Omega\cap B(0,R))} \leq C\varepsilon^\nu R^{1-\nu} \|\nabla
  u^*\|_{L^\infty(\widetilde \Omega)} +  \frac C
  {R^{d/q_2}} \|R^\varepsilon\|_{L^{q_2}(\widetilde\Omega)} + C \|R_1^\varepsilon\|_{L^\infty(\widetilde \Omega)}.
\end{equation}
Next, according to the definition \eqref{eq:reste} of $R^\varepsilon$, and using Assumption~\ref{H7} again, we have
\begin{multline}
  \label{eq:63}
  \|R^\varepsilon\|_{L^{q_2}(\widetilde\Omega)} \leq \|u^\varepsilon - u^*\|_{L^{q_2}(\widetilde \Omega)} +
  CR^{d/q_2}\left\|\varepsilon\sum_{j=1}^d
    w_{e_j}\left(\frac x \varepsilon\right)\partial_j u^*(x) \right\|_{L^\infty(\widetilde\Omega)} \\ \leq
  \|u^\varepsilon - u^*\|_{L^{q_2}(\widetilde \Omega)} + C \varepsilon^\nu R^{d/q_2+1-\nu}\|\nabla
  u^*\|_{L^\infty(\widetilde \Omega)}.
\end{multline}
Inserting \eqref{eq:63} into \eqref{eq:62}, we thus have
\begin{equation}
  \label{eq:64}
  \|R_2^\varepsilon\|_{L^\infty(\Omega\cap B(0,R))} \leq C\varepsilon^\nu R^{1-\nu} \|\nabla
  u^*\|_{L^\infty(\widetilde \Omega)} + \frac C {R^{d/q_2}}\|u^\varepsilon - u^*\|_{L^{q_2}(\widetilde \Omega)} + C \|R_1^\varepsilon\|_{L^\infty(\widetilde \Omega)}.
\end{equation}
Next, we bound $R_1^\varepsilon$. Denoting by $G^\varepsilon$ the Green function of the operator
$-\div\left(a\left(\frac x \varepsilon\right)\nabla \cdot\right)$ on $\widetilde \Omega$ with homogeneous Dirichlet
boundary conditions on $\widetilde \Omega$, we have, for any $x\in \widetilde\Omega$,
\begin{displaymath}
  R_1^\varepsilon(x) = -\int_{\widetilde\Omega} \nabla_y G^\varepsilon(x,y)\cdot H^\varepsilon(y)dy.
\end{displaymath}
Using the H\"older inequality and the estimate~\eqref{eq:estimation-H-L^p} of Lemma~\ref{lm2.7} (this is where we use
Assumption~\ref{H8}), we have
\begin{displaymath}
  |R_1^\varepsilon(x)|\leq C \varepsilon^\nu R^{1-\nu} \|\nabla_yG(x,\cdot)\|_{L^{q_1'}(\widetilde \Omega)} \|D^2 u^*\|_{L^{q_1}(\widetilde\Omega)}.
\end{displaymath}
Since $q_1>d$, we have $q_1'<\frac d {d-1}$, hence, using \cite[Equation (1.12)]{GruterWidman} and
Theorem~\ref{th:estimations_G}, 
\begin{multline}
  \label{eq:65}
  |R_1^\varepsilon(x)|\leq C \varepsilon^\nu R^{1-\nu} \left|\widetilde\Omega\right|^{\frac 1 {q_1'} -
    \frac{d-1}d}\left\|\nabla_y G(x,\cdot)\right\|_{L^{\frac d {d-1},\infty}(\widetilde\Omega)} \|D^2
  u^*\|_{L^{q_1}(\widetilde\Omega)} \\ \leq C \varepsilon^\nu R^{2-\frac d {q_1}-\nu} \|D^2
  u^*\|_{L^{q_1}(\widetilde\Omega)}.
\end{multline}
Collecting \eqref{eq:64} and \eqref{eq:65}, we have proved
\begin{multline}\label{eq:97}
  \left\|R^\varepsilon\right\|_{L^\infty(\Omega\cap B(x_0,R))} \leq CR^{-d/q_2}\|u^\varepsilon -
  u^*\|_{L^{q_2}(\Omega\cap B(x_0,4R))} + C\varepsilon^\nu R^{1-\nu} \|\nabla u^*\|_{L^\infty(\Omega\cap
    B(x_0,4R))} \\+ C \varepsilon^\nu R^{2-\frac d {q_1} - \nu} \|D^2 u^*\|_{L^{q_1}(\Omega\cap B(x_0,4R))}  .
\end{multline}
Next, we write
\begin{displaymath}
  u^\varepsilon(x) - u^*(x) = R^\varepsilon(x) + \varepsilon \sum_{j=1}^d w_{e_j}\left(\frac x
    \varepsilon\right) \partial_j u^*(x),
\end{displaymath}
which implies, using the triangle inequality and Assumption~\ref{H7}, 
\begin{displaymath}
  \left\| v^\varepsilon-v^*\right\|_{L^\infty(\Omega\cap B(x_0,R))}\leq
  \left\|R^\varepsilon\right\|_{L^\infty(\Omega\cap B(x_0,R))} + C \varepsilon\left(\frac R
    \varepsilon\right)^{1-\nu} \left\|\nabla u^*\right\|_{L^\infty(\Omega\cap B(x_0,R))}.
\end{displaymath}
Inserting \eqref{eq:97} into this estimate, we find \eqref{eq:cv_prelim}. 
\end{proof}

The following result is the generalization of \cite[Theorem 3.4]{KLSGreenNeumann} (with $q=\infty$ there) to the present setting. Here, the
proof is substantially different from \cite[Lemma 3.2]{KLSGreenNeumann}.

\begin{lemme}
  \label{KLSth3.4}
Under the assumptions of Theorem~\ref{th:cv_Green}, let $q>d$, $x_0\neq y_0\in \Omega$, $R = |x_0-y_0|/16$. Assume
that $f\in C^\infty_c(\Omega\cap B(y_0,4R))$, and that $u^\varepsilon$ and $u^*$ are solutions to 
  \begin{displaymath}
\left\{
      \begin{aligned}
        &-\div\left(a\left(\frac x {\varepsilon}\right) \nabla u^\varepsilon\right) = -\div(a^*\nabla u^*) =f & \text{
          in } \Omega, \\
        &u^\varepsilon = u^*= 0 & \text{ on } \partial\Omega.
      \end{aligned}
\right.  
  \end{displaymath}
Then,
\begin{equation}
  \label{eq:67}
  \left\|u^\varepsilon - u^* \right\|_{L^\infty(\Omega\cap B(x_0,R))} \leq C R^{2-\frac d q -\nu} \varepsilon^\nu
  \|f\|_{L^q(\Omega)}, 
\end{equation}
where $C$ depends only on the coefficient $a$, $q$ and $\Omega$.
\end{lemme}
\begin{proof} Due to translation invariance, we may assume that $y_0 = 0$. 
  We apply Lemma~\ref{KLSlem3.2} with $q_1 = q$. Hence, $u^\varepsilon-u^*$ satisfies (\ref{eq:cv_prelim}), that is,
\begin{multline}\label{eq:68}
  \left\|u^\varepsilon - u^*\right\|_{L^\infty(\Omega\cap B(x_0,R))} \leq CR^{-\frac d{q_2}}\|u^\varepsilon -
  u^*\|_{L^{q_2}(\Omega\cap B(x_0,4R))} + C\varepsilon^\nu R^{1-\nu} \|\nabla u^*\|_{L^\infty(\Omega\cap
    B(x_0,4R))} \\+ C \varepsilon^\nu R^{2-\frac d {q} - \nu} \|D^2 u^*\|_{L^{q}(\Omega\cap B(x_0,4R))},
\end{multline}
for any $q_2>1$. We fix $q_2<2$, and we are going to estimate separately each term of the right-hand side of (\ref{eq:68}).

\noindent\underline{Step 1: bound on $\|\nabla u^*\|_{L^\infty}$.} Denoting by $G^*$ the Green function of the operator
$-\div\left(a^*\nabla\cdot\right)$ with homogeneous Dirichlet boundary conditions on $\Omega$, we have
  \begin{displaymath}
    \forall x\in \Omega, \quad \nabla u^*(x) = \int_\Omega \nabla_x G^*(x,y)f(y) dy = \int_{\Omega\cap B(0,4R)}
    \nabla _x G^*(x,y)f(y)dy. 
  \end{displaymath}
Hence, $|\nabla u^*(x)|\leq \|\nabla_x G^*(x,\cdot)\|_{L^{q'}(\Omega\cap B(0,4R))} \|f\|_{L^q(\Omega)}.$ Applying \cite[Theorem 3.3 (iv)]{GruterWidman}, we have
\begin{displaymath}
  \|\nabla_x G^*\left(x,\cdot\right)\|_{L^{q'}(\Omega\cap B(0,4R))} \leq C\left(\int_{B(0,4R)}\frac 1
    {|x-y|^{q'(d-1)}}dy\right)^{1/q'}  \leq
  \begin{cases}
    C R^{\frac d {q'} -d +1} & \text{ if } |x|< 8R, \\
    C R^{\frac d {q'}}|x|^{-d+1} &\text{ if } |x|\geq 8R.
  \end{cases}
\end{displaymath}
Hence,
\begin{equation}
  \label{eq:79}
  |\nabla u^*(x)|\leq C \|f\|_{L^q(\Omega)} \frac{R^{d-\frac d q}}{\max\left(R^{d-1},|x|^{d-1}\right)} \ .
\end{equation}
In particular, we have
\begin{equation}
  \label{eq:69}
  \|\nabla u^*\|_{L^\infty(\Omega)} \leq C R^{1-\frac d q}\|f\|_{L^q(\Omega)},
\end{equation}
where $C$ depends only on $a^*$ and $\Omega$. 

\noindent\underline{Step 2: bound on $\|D^2u^*\|_{L^q}.$} According to standard elliptic regularity results (see for
instance \cite[Lemma 9.17]{GT}), we have 
\begin{equation}
  \label{eq:70}
  \|D^2 u^*\|_{L^q(\Omega)} \leq C \|f\|_{L^q(\Omega)},
\end{equation}
where $C$ depends only on $a^*$ and $\Omega$. In addition, using the Green function representation again, 
\cite[Theorem 3.3 (vi)]{GruterWidman}, and an argument similar to the proof of \eqref{eq:79}, we have, if $|x|>8R$,
\begin{equation}
  \label{eq:71}
  |D^2 u^*(x)|\leq \left|\int_{\Omega} \nabla_x^2 G^*(x,y)f(y)dy\right| \leq C \int_{\Omega\cap B(0,4R)} \frac 1
  {|x-y|^d} |f(y)|dy \leq C \frac{R^{d-\frac d q}}{|x|^d}\|f\|_{L^q(\Omega)}.
\end{equation}
Pointing out that $|x|\geq R$ for all $x\in B(x_0,4R)$,
this implies, using the H\"older inequality,
\begin{multline}
  \label{eq:98}
  \left\|D^2 u^*\right\|_{L^q(\Omega\cap B(x_0,4R))} \leq C R^{d-d/q}\frac 1 {R^d} R^{d/q'} |B(x_0,4R)|^{d/q}
  \|f\|_{L^q(\Omega)} = CR^{d-2d/q} \|f\|_{L^q(\Omega)}\\ \leq C \left(\diam(\Omega)\right)^{d-2d/q} \|f\|_{L^q(\Omega)}.
\end{multline}

\noindent\underline{Step 3: bound on $\|u^\varepsilon - u^*\|_{L^q(\Omega\cap B(x_0,4R))}$.} As in the proof
of Lemma~\ref{KLSlem3.2}, we define $R^\varepsilon$ by \eqref{eq:reste}, and write $R^\varepsilon = R_1^\varepsilon
+ R_2^\varepsilon$, where $R_1^\varepsilon$ and $R_2^\varepsilon$ are solutions to \eqref{eq:72} and \eqref{eq:73},
respectively (with $\widetilde \Omega = \Omega$), and $H^\varepsilon$ is defined by \eqref{eq:H-epsilon}. 
Mutliplying the first line of~\eqref{eq:72} by $R_1^\varepsilon$ and integrating, we have
\begin{equation}\label{eq:77}
  \|\nabla R_1^\varepsilon\|_{L^2(\Omega)} \leq C \|H^\varepsilon\|_{L^2(\Omega)},
\end{equation}
where $C$ depends only on the ellipticity constant of the coefficient $a$. We claim that
\begin{equation}
  \label{eq:74}
  \|H^\varepsilon\|_{L^2(\Omega)} \leq C \varepsilon^\nu R^{\frac d 2 - \frac d q +1-\nu}\|f\|_{L^q(\Omega)}.
\end{equation}
We first deal with $\|H^\varepsilon\|_{L^2(\Omega\setminus B(0,8R))}$, then with 
$\|H^\varepsilon\|_{L^2(\Omega\cap B(0,8R))}$. Using Assumptions~\ref{H7} and \ref{H8}, we have, for all $x\in \Omega$, 
\begin{multline*}
  \left|H^\varepsilon(x)\right|\leq \varepsilon \|a\|_{L^\infty} \sup_{1\leq j\leq d} \left|w_{e_j}\left(\frac x
      \varepsilon\right)\right| \,|D^2 u^*(x)|+ \varepsilon \left|B\left(\frac x \varepsilon\right)\right| \, \left|D^2 u^*(x)\right| \leq C
  \varepsilon \left(\frac{|x|}\varepsilon\right)^{1-\nu} |D^2 u^*(x)| \\
\leq C \varepsilon^\nu |x|^{1-\nu} |D^2 u^*(x)|.
\end{multline*}
We then compute the $L^2$ norm of $H^\varepsilon$ on $\Omega\setminus B(0,8R)$, and use \eqref{eq:71}, together
with $|x|\geq R$:
\begin{multline}\label{eq:75}
  \left\|H^\varepsilon\right\|_{L^2(\Omega\setminus B(0,8R))} \leq C \varepsilon^\nu \left(\int_{\Omega\setminus
      B(0,8R)} |x|^{2(1-\nu)} |D^2 u^*(x)|^2dx \right)^{1/2} \\ \leq C \varepsilon^\nu \left(\int_{\Omega\setminus
      B(0,8R)} \frac 1 {|x|^{2d+2\nu-2}}R^{2d-\frac{2d}q} dx \right)^{1/2} \|f\|_{L^q(\Omega)} \leq C
  \varepsilon^\nu R^{1-\nu + \frac d 2 - \frac d q} \|f\|_{L^q(\Omega)}.
\end{multline}
In addition, successively using Lemma~\ref{lm2.7} (with $q=2$ there), the H\"older inequality, and \cite[Lemma 9.17]{GT},
\begin{equation}
  \label{eq:76}
   \|H^\varepsilon\|_{L^2(\Omega\cap B(0,8R))} \leq C \varepsilon^\nu R^{1-\nu}\|D^2 u^*\|_{L^2(\Omega\cap
     B(0,8R))} \leq C \varepsilon^\nu R^{\frac d 2 - \frac d q + 1-\nu}\|f\|_{L^q(\Omega)}.
\end{equation}
Collecting \eqref{eq:75} and \eqref{eq:76}, we infer \eqref{eq:74}. Inserting \eqref{eq:74} into \eqref{eq:77}, we
thus have $\|\nabla R_1^\varepsilon\|_{L^2(\Omega)} \leq C \varepsilon^\nu R^{\frac d 2 - \frac d q
  +1-\nu}\|f\|_{L^q(\Omega)}$. Hence, using the H\"older inequality again and Sobolev embeddings,
\begin{multline}
  \label{eq:78}
  \|R_1^\varepsilon\|_{L^{q_2}(\Omega\cap B(0,4R))} \leq C R^{\frac d {q_2} - \left(\frac d 2 -1\right)}
  \|R_1^\varepsilon\|_{L^{\frac{2d}{d-2}}(\Omega\cap B(0,4R))} \leq C R^{\frac d {q_2} +1-\frac d 2} \|\nabla
  R_1^\varepsilon\|_{L^2(\Omega)} \\ \leq C \varepsilon^\nu R^{2+\frac d {q_2} - \frac d q -\nu}\|f\|_{L^q(\Omega)}.
\end{multline}
We estimate $R_2^\varepsilon$. Using the maximum principle,
 we have
\begin{displaymath}
  \|R_2^\varepsilon\|_{L^\infty(\Omega)} \leq \left\|\varepsilon\sum_{j=1}^d w_{e_j}\left(\frac \cdot
      \varepsilon\right) \partial_j u^*\right\|_{L^\infty(\partial\Omega)}.
\end{displaymath}
This estimate, together with \eqref{eq:79} and Assumption~\ref{H7}, imply
\begin{displaymath}
  \|R_2^\varepsilon\|_{L^\infty(\Omega)}\leq C\varepsilon^\nu R^{d-\frac d q}\|f\|_{L^q(\Omega)} \sup_{x\in\partial
    \Omega} \left\{\frac{|x|^{1-\nu}}{\max\left(R^{d-1},|x|^{d-1}\right)} \right\} \leq C\varepsilon^\nu
  R^{d-\frac d q}\|f\|_{L^q(\Omega)} R^{2-d-\nu}.
\end{displaymath}
Thus,
\begin{equation}
  \label{eq:80}
  \|R_2^\varepsilon\|_{L^\infty(\Omega)} \leq C\varepsilon^\nu R^{2-\nu-\frac d q}\|f\|_{L^q(\Omega)} .
\end{equation}
We next bound $u^\varepsilon-u^*$. Applying the triangle inequality, 
\begin{multline}\label{eq:82}
  \|u^\varepsilon - u^*\|_{L^{q_2}(\Omega\cap B(x_0,4R))} \leq \left\|\varepsilon\sum_{j=1}^d w_{e_j}\left(\frac \cdot
      \varepsilon\right) \partial_j u^* \right\|_{L^{q_2}(\Omega\cap B(x_0,4R))} \\+ \left\| R_1^\varepsilon\right\|_{L^{q_2}(\Omega\cap B(x_0,4R))} + \left\|R_2^\varepsilon \right\|_{L^{q_2}(\Omega\cap B(x_0,4R))}.
\end{multline}
The first term is bounded using Assumption~\ref{H7} and \eqref{eq:69}:
\begin{equation}
  \label{eq:81}
  \left\|\varepsilon\sum_{j=1}^d w_{e_j}\left(\frac \cdot
      \varepsilon\right) \partial_j u^* \right\|_{L^{q_2}(\Omega\cap B(x_0,4R))} \leq C \varepsilon^\nu R^{2+\frac
    d {q_2} - \frac d q -\nu} \|f\|_{L^q(\Omega)}.
\end{equation}
Hence, inserting \eqref{eq:78}, \eqref{eq:80}, \eqref{eq:81} into \eqref{eq:82}, we infer
\begin{equation}
  \label{eq:83}
  \|u^\varepsilon - u^*\|_{L^{q_2}(\Omega\cap B(x_0,4R))} \leq C \varepsilon^\nu R^{2+\frac d {q_2} - \frac d q
    -\nu}  \|f\|_{L^q(\Omega)}.
\end{equation}
Finally, we collect \eqref{eq:68}, \eqref{eq:69}, \eqref{eq:98} and \eqref{eq:83}, which proves \eqref{eq:67}.
\end{proof}

We are now in position to prove Theorem~\ref{th:cv_Green}.
\begin{proof}[Proof of Theorem~\ref{th:cv_Green}]
  Let $q>d$, $x_0,y_0\in \Omega$, $R= |x_0-y_0|/16$ and $f\in C_c^\infty(\Omega\cap B(y_0,4R))$. We apply
  Lemma~\ref{KLSth3.4}. We have
  \begin{displaymath}
    u^\varepsilon(x) = \int_\Omega G^\varepsilon(x,y)f(y)dy \quad \text{and} \quad    u^*(x) = \int_\Omega G^*(x,y)f(y)dy.
  \end{displaymath}
Since $q>d,$ we may apply inequality \eqref{eq:67}. This gives
\begin{displaymath}
  \left|\int_\Omega \left(G^\varepsilon(x,y) - G^*(x,y)\right)f(y)dy\right|\leq C \varepsilon^\nu R^{2-\frac d q -
    \nu} \|f\|_{L^q(\Omega\cap B(y_0,R))}.
\end{displaymath}
Thus, a duality argument allows to prove
\begin{equation}
  \label{eq:84}
  \| G^\varepsilon(x,\cdot) - G^*(x,\cdot)\|_{L^{q'}(\Omega\cap B(y_0,4R))} \leq C\varepsilon^\nu R^{2-\frac d q -\nu}.
\end{equation}
Moreover, $G^\varepsilon$ and $G^*$ satisfy
\begin{displaymath}
  \left\{
    \begin{aligned}
      -\div_y\left( a^T \left(\frac y \varepsilon\right) \nabla_y G^\varepsilon(x_0,y)\right) = 0 && \text{ in } &
      \Omega\cap B(y_0,4R), \\
      -\div_y\left( \left(a^*\right)^T  \nabla_y G^*(x_0,y)\right) = 0 && \text{ in } &
      \Omega\cap B(y_0,4R), \\
      G^\varepsilon(x_0,\cdot) = G^*(x_0, \cdot) = 0 &&\text{ on }& \left(\partial\Omega\right)\cap \overline B(y_0,4R).
    \end{aligned}
\right.
\end{displaymath}
Hence, we may apply Lemma~\ref{KLSlem3.2} with $q_2 = q'$. This implies
\begin{multline*}
  |G^\varepsilon(x_0,y_0) - G^*(x_0,y_0)|\leq C R^{-\frac d {q'}}\|G^\varepsilon(x_0, \cdot) - G^*(x_0,\cdot)
  \|_{L^{q'}(\Omega\cap B(y_0,4R))} \\+ C \varepsilon^\nu R^{1-\nu} \|\nabla_y G^*(x_0,\cdot)\|_{L^\infty(\Omega\cap
    B(y_0,4R))} + C \varepsilon^\nu R^{2-\frac d q -\nu}\|D^2_y G^*(x_0,\cdot)\|_{L^q(\Omega\cap
    B(y_0,4R))}\ .
\end{multline*}
Applying once again \cite[Theorem 3.3]{GruterWidman} to $G^*$, we have $\|\nabla_y G^*(x_0,\cdot)\|_{L^\infty(\Omega\cap
    B(y_0,4R))} \leq C R^{1-d}$ and $\|D^2_y G^*(x_0,\cdot)\|_{L^q(\Omega\cap
    B(y_0,4R))} \leq C R^{-\frac d {q'}}$. Thus, using \eqref{eq:84}, we get
  \begin{displaymath}
    |G^\varepsilon(x_0,y_0) - G^*(x_0,y_0)|\leq C \varepsilon^\nu R^{2-d-\nu},
  \end{displaymath}
which concludes the proof, since $16R = |x_0-y_0|$. 
\end{proof}

Next, we prove the following result, which is a consequence of Theorem~\ref{th:cv_Green}, and is the generalization of \cite[Theorem 3.4]{KLSGreenNeumann} to the present setting.

\begin{corollaire}
  \label{cor:1.14}
  Assume that the matrix-valued coefficients $a$ and $a^T$ satisfy Assumptions \ref{H1} through \ref{H6}, and
  \ref{H7}-\ref{H8} for some $\nu>0$. Let $\Omega$ be a bounded $C^{2,1}$ domain and $q\in [1,+\infty[$. Then there exists a constant
  $C>0$ depending only on $a$, $\nu$, $\Omega$ and $q$, such that for any $f\in L^q(\Omega)$, if $u^\varepsilon$
  and $u^*$ are solution to \eqref{eq:equation} and \eqref{eq:homog_periodique}, respectively, then 
  \begin{displaymath}
    \|u^\varepsilon - u^*\|_{L^s(\Omega)} \leq C\varepsilon^\nu \|f\|_{L^q(\Omega)},
  \end{displaymath}
where 
\begin{displaymath}
  \frac 1 q - \frac{2-\nu}d = \frac 1 s, \quad\text{if}\quad \frac 1 q > \frac{2-\nu}d,
\end{displaymath}
and
\begin{displaymath}
  s= +\infty \quad\text{if} \quad \frac 1 q < \frac{2-\nu}d .
\end{displaymath}
\end{corollaire}
\begin{proof}
First, assume that $\frac 1 q > \frac{2-\nu}d$. Since the function
  $g$ defined by $g(x) = |x|^{2-d-\nu}$ satisfies $g\in L^{d/(d-2+\nu),\infty}$, and since $u^\varepsilon-u^*$ satisfies
  \begin{displaymath}
    u^\varepsilon(x) -u^*(x) = \int_\Omega \left(G^\varepsilon(x,y) - G^*(x,y)\right)f(y)dy,
  \end{displaymath}
we use Theorem~\ref{th:cv_Green} and a simple application of Young-O'Neil inequality \cite{oneil,yap}, which gives
\begin{displaymath}
  \left\|u^\varepsilon - u^*\right\|_{L^s(\Omega)} \leq C \varepsilon^\nu \bigl\|\,g*|f|\,\bigr\|_{L^{s,q}(\Omega)} \leq
  C\varepsilon^\nu \left\|g\right\|_{L^{d/(d-2+\nu),\infty}(\Omega)}\|f\|_{L^q(\Omega)},
\end{displaymath}
which proves the result. The case $s=+\infty$ is treated by a similar argument. 
\end{proof}

\section{Proof of the main results}
\label{sec:proof}

\subsection{Proof of Theorem~\ref{th:general}}

We give in this section the

\begin{proof}[Proof of Theorem~\ref{th:general}]
  We first prove \eqref{eq:cv_L2_general}. Applying Corollary~\ref{cor:1.14}, we clearly have
  \begin{displaymath}
    \|R^\varepsilon\|_{L^2(\Omega)} \leq C\varepsilon^\nu \|f\|_{L^2(\Omega)} + \varepsilon \left\|\sum_{j=1}^d
      w_{e_j}\left(\frac \cdot \varepsilon\right) \partial_j u^*\right\|_{L^2(\Omega)}.
  \end{displaymath}
Hence, using Assumption~\ref{H7} and the fact that $\|\nabla u^*\|_{L^2(\Omega)}\leq C \|f\|_{L^2(\Omega)}$, we
deduce \eqref{eq:cv_L2_general}.

\medskip

Next, we prove \eqref{eq:cv_H1_general}. For this purpose, we write again $R^\varepsilon = R_1^\varepsilon +
  R_2^\varepsilon$, where $R_1^\varepsilon$ and $R_2^\varepsilon$ are defined by \eqref{eq:72} and \eqref{eq:73},
  respectively (with $\widetilde \Omega = \Omega$.)
Multiplying the first line of \eqref{eq:72} by $R_1^\varepsilon$ and integrating by parts, we have $\|\nabla
R_1^\varepsilon\|_{L^2(\Omega)}\leq C\|H^\varepsilon\|_{L^2(\Omega)}$. Hence, applying Lemma~\ref{lm2.7}, we have
\eqref{eq:estimation-H-L^p}, which implies
\begin{equation}
  \label{eq:87}
  \|\nabla R_1^\varepsilon\|_{L^2(\Omega)} \leq C \varepsilon^\nu \|D^2 u^*\|_{L^2(\Omega)} \leq C \varepsilon^\nu \|f\|_{L^2(\Omega)}.
\end{equation}
The right-most estimate is a consequence of standard elliptic regularity estimates \cite[Lemma 9.17]{GT}. Next, we
apply the Caccioppoli inequality (actually, we need to cover $\Omega_1\subset\subset \Omega$ by balls $B_r(x_i)$ such
that $B_{2r}(x_i)\subset \Omega$ for each $i$, and apply the Caccioppoli inequality on each of theses balls), getting
\begin{displaymath}
  \|\nabla R_2^\varepsilon\|_{L^2(\Omega_1)} \leq C \left\|R_2^\varepsilon\right\|_{L^2(\Omega)} \leq C \|R_1^\varepsilon\|_{L^2(\Omega)} +
  C \|R^\varepsilon\|_{L^2(\Omega)} \leq C \|\nabla R_1^\varepsilon\|_{L^2(\Omega)} +
  C \|R^\varepsilon\|_{L^2(\Omega)},
\end{displaymath}
where we applied the Poincar\'e inequality to $R_1^\varepsilon$. The constant $C$ in the above inequality only depends
on $\Omega_1$, $\Omega$, and the coefficient $a$. Using \eqref{eq:cv_L2_general} and \eqref{eq:87}, we prove
\eqref{eq:cv_H1_general}.

\medskip

We now turn to the proof of \eqref{eq:cv_W1p_general}. We fix $\Omega_2$ such that $\Omega_1\subset\subset\Omega_2
\subset\subset \Omega$. We cover $\Omega_1$ by balls $B_r(x_j)$ such that $B_{2r}(x_j)\subset \Omega_2$ for all
$j$. Applying Proposition~\ref{pr:1} to $R^\varepsilon$, we have
\begin{displaymath}
  \|\nabla R^\varepsilon\|_{L^q(\Omega_1)} \leq C\|H^\varepsilon\|_{L^q(\Omega)}+C \|\nabla R^\varepsilon\|_{L^2(\Omega_2)}
\end{displaymath}
Hence, using \eqref{eq:cv_H1_general} and \eqref{eq:estimation-H-L^p} again, this implies
\begin{displaymath}
  \|\nabla R^\varepsilon\|_{L^q(\Omega_1)} \leq C \varepsilon^\nu \|D^2 u^*\|_{L^q(\Omega)} +
  C\varepsilon^\nu\|f\|_{L^2(\Omega)}
\end{displaymath}
Here again, elliptic regularity \cite[Lemma 9.17]{GT} implies $\|D^2u^*\|_{L^q(\Omega)} \leq \|f\|_{L^q(\Omega)}$,
and we conclude using the H\"older inequality. 

\medskip

Finally, we prove \eqref{eq:cv_lipschitz_general}. We assume $f\in C^{0,\beta}(\Omega)$. We first assume $\beta
\leq \alpha$. Here again, we define $\Omega_2$ such that $\Omega_1\subset\subset\Omega_2
\subset\subset \Omega$. We cover $\Omega_1$ by balls $B_r(x_j)$ such that $B_{2r}(x_j)\subset \Omega_2$ for all
$j$. Applying Proposition~\ref{pr:2} to $v^\varepsilon = R^\varepsilon$, we find
\begin{displaymath}
  \|\nabla R^\varepsilon\|_{L^\infty(\Omega_1)} \leq C \|\nabla R^\varepsilon\|_{L^2(\Omega_2)} +
  C\varepsilon^\beta [H^\varepsilon]_{C^{0,\beta}(\Omega)} + C \ln\left(2+\varepsilon^{-1}\right) \|H^\varepsilon\|_{L^\infty(\Omega)}.
\end{displaymath}
We apply \eqref{eq:cv_H1_general}, \eqref{eq:estimation-H-L^p} and \eqref{eq:estimation-H-Holder}, whence
\begin{displaymath}
  \|\nabla R^\varepsilon\|_{L^\infty(\Omega_1)} \leq C \varepsilon^\nu \|f\|_{L^2(\Omega_2)} +
  C\varepsilon^{\nu+\beta} [D^2u^*]_{C^{0,\beta}(\Omega)} + C \varepsilon^\nu\ln\left(2+\varepsilon^{-1}\right)
  \|D^2 u^*\|_{L^\infty(\Omega)}.
\end{displaymath}
Here again, we apply standard elliptic estimates \cite[Corollary 8.36]{GT}, thereby proving
\eqref{eq:cv_lipschitz_general}. 

We assume now that $\beta>\alpha$. In particular, we have $f\in C^{0,\alpha}(\Omega)$. Thus, we may apply the above
result with $\beta = \alpha$, and we have
\begin{displaymath}
   \|\nabla R^\varepsilon\|_{L^\infty(\Omega_1)} \leq C\varepsilon^\nu \ln\left(2+\varepsilon^{-1}\right)
   \|f\|_{C^{0,\alpha}(\Omega)} \leq C\varepsilon^\nu \ln\left(2+\varepsilon^{-1}\right)
   \|f\|_{C^{0,\beta}(\Omega)},
\end{displaymath}
which completes the proof.
\end{proof}

\subsection{Application to local perturbations of periodic problems: proof of Theorem~\ref{th:defaut}}

We prove here that the setting defined by \eqref{eq:aper+tildea}, \eqref{eq:hyp1} is covered by
Theorem~\ref{th:general} with $\nu = \nu_r$ defined by \eqref{eq:nu_r}, thereby proving
Theorem~\ref{th:defaut}. First, we recall that \cite{BLLfutur1} (see also \cite{BLLcpde}) shows that
in such a setting, the corrector equation \eqref{eq:correcteur} has a solution $w_p$ which reads as
\eqref{eq:correcteur-decomposition}, where $\widetilde w_p$ satisfies
\begin{equation}
  \label{eq:estimation_w_tilde_1}
  \text{if}\quad r>1, \quad \nabla \widetilde w_p\in L^q(\RR^d), \quad \forall q\in [r,+\infty[,
\end{equation}
\begin{equation}
  \label{eq:estimation_w_tilde_2}
  \text{if}\quad r=1, \quad \nabla \widetilde w_p\in L^q(\RR^d), \quad \forall q\in ]1,+\infty[,
\end{equation}
and with the property
\begin{equation}
  \label{eq:estimation_w_tilde_3}
  \text{if}\quad r<d, \quad \widetilde w_p\in L^\infty(\RR^d).
\end{equation}

\begin{proposition}\label{pr:defaut1}
  Assume that the matrix-valued coefficient $a$ satisfies \eqref{eq:aper+tildea} and \eqref{eq:hyp1}, with $r\neq d$. Then there
  exists a constant $C>0$ depending only on $a$ such that
  \begin{equation}
    \label{eq:borne_correcteur_defaut}
    \forall p\in \RR^d, \quad \forall x\in \RR^d,\quad \forall y\in \RR^d, \quad |w_p(x) - w_p(y)|\leq C|p|\, |x-y|^{1-\nu_r},
  \end{equation}
where $\nu_r$ is defined by \eqref{eq:nu_r}.
\end{proposition}

\begin{remarque}
  In Proposition~\ref{pr:defaut1}, the case $r=d$ is not covered. However, since in fact $\widetilde a\in L^r\cap
  L^\infty$, this case can be addressed using the fact that $\widetilde a\in L^r$ for any $r>d$. 
\end{remarque}

\begin{proof}
  Since $p\mapsto w_p$ is a linear map, it is sufficient to prove \eqref{eq:borne_correcteur_defaut} in the case
  $|p|=1$.  First, elliptic regularity \cite[Theorem 8.32]{GT} implies that $w_{p,per}\in C^{1,\alpha}_{\rm unif}(\RR^d)$,
  hence it clearly satisfies \eqref{eq:borne_correcteur_defaut}. Therefore, we only prove that $\widetilde w_p$
  satisfies \eqref{eq:borne_correcteur_defaut}. 

If $r<d$, $\nu_r=1$, and \eqref{eq:borne_correcteur_defaut} is a direct consequence of
\eqref{eq:estimation_w_tilde_3}.

If $r>d$, we apply Morrey's Theorem \cite[Theorem 4.10]{EvansGariepy} to $\widetilde w_p$:
\begin{displaymath}
  |\widetilde w_p(x) - \widetilde w_p(y)|\leq \left\|\nabla \widetilde w_p \right\|_{L^r(\RR^d)} |x-y|^{1-\frac d
    r} = \left\|\nabla \widetilde w_p \right\|_{L^r(\RR^d)} |x-y|^{1-\nu_r}
\end{displaymath}
Applying the triangle inequality, \eqref{eq:borne_correcteur_defaut} is proved.
\end{proof}

We now prove that a potential $B$ defined by (\ref{eq:B-equation-2}) exists and has suitable properties in the
present setting. 
\begin{lemme}\label{lm:existence_B}
  Assume that $1<q<+\infty,$  and that $\widetilde M = \left(\widetilde M_k^i\right)_{1\leq i,k\leq d}\in L^q(\RR^d)$ satisfies
  \begin{displaymath}
    \forall k\in \{1,\dots,d\}, \quad \div\left(\widetilde M_k\right) = 0.
  \end{displaymath}
Then, the potential $\widetilde B_k^{ij}$ defined by
\begin{equation}\label{eq:1}
  \widetilde B_k^{ij}(x) = \int_{\RR^d} \left(\frac 1 {d\omega_d}\, \frac{x_i-y_i}{|x-y|^d}\,\widetilde M_k^j(y) - \frac 1 {d\omega_d}
    \,\frac{x_j-y_j}{|x-y|^d}\, \widetilde M_k^i(y)\right)dy,
\end{equation}
where the constant $\omega_d$ is the surface of the unit sphere in $\RR^d$,
satisfies $\nabla \widetilde B\in L^q(\RR^d)$, and (\ref{eq:B-equation-2}), hence
(\ref{eq:Mik})-(\ref{eq:B-antisymetrique})-(\ref{eq:B-equation}). In addition, there exists a constant depending
on $d$ and $q$ only such that
\begin{equation}\label{eq:2}
  \left\|\nabla \widetilde B\right\|_{L^q(\RR^d)} \leq C \left\|\widetilde M\right\|_{L^q(\RR^d)}.
\end{equation}
Finally, if $q<d$ and if $\widetilde M\in L^\infty(\RR^d)$, then $\widetilde B \in L^\infty(\RR^d)$, and there exists a
constant depending only on $d$ and $q$ such that
\begin{equation}
  \label{eq:B_borne}
  \left\|\widetilde B \right\|_{L^\infty(\RR^d)} \leq C \left( \left\|\widetilde M\right\|_{L^q(\RR^d)} + \left\|\widetilde M\right\|_{L^\infty(\RR^d)}\right).
\end{equation}

\end{lemme}

\begin{proof}
First, it is clear that (\ref{eq:1}) is a well-defined function if $M$ has compact support. Next, we consider the
operator $T$, which to $\widetilde M$ associates $\nabla \widetilde B$. Moreover, 
\eqref{eq:B-antisymetrique}-\eqref{eq:B-equation} are clearly satisfied by $\widetilde B$, hence, we
have~\eqref{eq:B-equation-2}. Multiplying it by $B$ and integrating by parts, we have
\begin{displaymath}
  \int_{\RR^d} |\nabla B|^2 = \int_{\RR^d} - M_k^i \partial_j B_k^{ij} + M_k^j\partial_i B_k^{ij} \leq
  \|M\|_{L^2(\RR^d)} \, \|\nabla B\|_{L^2(\RR^d)}.
\end{displaymath}
Hence, a density argument allows to define it as a
continuous operator from $L^2(\RR^d)$ to itself. Furthermore, $T$ is a Calderon-Zygmund operator (see \cite[Def. 1, p
224]{Meyer2}). Hence, (\ref{eq:2}) holds. 

It remains to prove \eqref{eq:B_borne}. We split the integral in \eqref{eq:1} into the integral over $|x-y|<1$ and
the integral over $|x-y|>1$, and find
\begin{displaymath}
  \left|\widetilde B (x) \right|\leq C \left| \widetilde M * \left(\frac 1{ |x|^{d-1}} \mathbf{1}_{|x|<1}\right)\right|(x)
  + C\left| \widetilde M * \left(\frac 1{ |x|^{d-1}} \mathbf{1}_{|x|>1}\right)\right|(x).
\end{displaymath}
Hence, applying the H\"older inequality,
\begin{displaymath}
  \left\|\widetilde B\right\|_{L^\infty(\RR^d)} \leq C \left\|\widetilde M \right\|_{L^\infty(\RR^d)}\left\|\frac 1
    {|x|^{d-1}}\right\|_{L^1(B(0,1)} + C \left\|\widetilde M \right\|_{L^q(\RR^d)}\left\|\frac 1
    {|x|^{d-1}}\right\|_{L^{q'}(B(0,1)^C}.
\end{displaymath}
We point out that, on the one hand, $|x|^{1-d}\in L^1(B(0,1))$, and on the other hand, since $q<d$, $q'>d/(d-1)$,
whence $|x|^{1-d}\in L^{q'}(B(0,1)^C)$. We have thus proved \eqref{eq:B_borne}.
\end{proof}

\begin{proposition}\label{pr:defaut2}
  Assume that the matrix-valued coefficient $a$ satisfies \eqref{eq:aper+tildea} and \eqref{eq:hyp1} for some $r\geq 1$. 
  Let 
\begin{displaymath}
  M_k^i(x) = a_{ik}^* - \sum_{j=1}^d a_{ij}(x) \left(\delta_{jk} + \partial_j w_{e_k}(x)\right)
\end{displaymath}
be defined by \eqref{eq:Mik}.
Then there exists $B_k^{ij}$, $1\leq i,j,k\leq d,$ solution to \eqref{eq:B-antisymetrique}-\eqref{eq:B-equation},
that is,
\begin{displaymath}
  \forall i,j,k\in \{1,\dots,d\}, \quad  B_k^{ij} = - B_k^{ji},
 \qquad \sum_{i=1}^d \partial_i B_k^{ij} = M_k^j.
\end{displaymath}
Moreover, if $r\neq d$, then there exists $C>0$ such that
\begin{equation}
  \label{eq:estimation_B}
  \forall x\in \RR^d, \quad\forall y\in \RR^d, \quad |B(x)-B(y)|\leq C |x-y|^{1-\nu_r}.
\end{equation}
\end{proposition}
\begin{proof}
  We define $B = B^{per} + \widetilde B$, where $B^{per}$ is the periodic solution to
  \begin{displaymath}
    \left(B^{per}\right)_k^{ij} = -\left(B^{per}\right)_k^{ij}, \qquad  \sum_{i=1}^d \partial_i \left(B^{per}\right)_k^{ij} = a_{jk}^* -
    \sum_{i=1}^d a_{ji}^{per} \left(\delta_{ik} + \partial_i w^{per}_{e_k}\right) := \left(M^{per}\right)_k^{j}.
  \end{displaymath}
This solution is proved to exist in \cite[pages 6-7]{JKO}. In addition, $B^{per}$ is solution to 
\begin{displaymath}
  \Delta \left(B^{per}\right)_k^{ij} = \partial_i \left(M^{per}\right)_k^j - \partial_j \left(M^{per}\right)_k^i.
\end{displaymath}
Our Assumption~\ref{H2} and classical elliptic regularity (applied to $w_{p,per}$) show that $\left(M^{per}\right)_k^j$ is in
$C^{0,\alpha}_{\rm unif}(\RR^d)$. Hence, still using elliptic regularity \cite[Corollary 8.32]{GT}, we have $\nabla
B^{per} \in C^{0,\alpha}_{\rm unif}(\RR^d)$. Arguing as in the proof of Proposition~\ref{pr:defaut1}, we obtain that
$B^{per}$ satisfies \eqref{eq:estimation_B}.

We now turn to $\widetilde B$. In order to define it, we first set, for all $j,k$,
\begin{equation}
  \label{eq:def_M_tilde}
  \widetilde M_k^{j}= -\sum_{i=1}^d \widetilde a_{ji} \left(\delta_{ik} + \partial_i w_{e_k}\right) - \sum_{i=1}^d
    a^{per}_{ij} \partial_i\widetilde w_{e_k} .  
\end{equation}
In view of \eqref{eq:estimation_w_tilde_1} and \eqref{eq:estimation_w_tilde_2}, we have $\widetilde M \in
L^q(\RR^d)$, for any $q\in ]r,+\infty[$, with $q=r$ allowed if $r>1$. Hence, $\widetilde M$ satisfies the
assumptions of Lemma~\ref{lm:existence_B}, hence there exists $\widetilde B$, defined by \eqref{eq:1}. We have
$\nabla \widetilde B \in L^q(\RR^d)$, and one easily proves that $\tilde B$ is a solution to 
\begin{equation}\label{eq:88}
    \widetilde B_k^{ij} = -\widetilde B_k^{ij}, \qquad  \sum_{i=1}^d \partial_i \widetilde B_k^{ij} = -
    \sum_{i=1}^d \widetilde a_{ji} \left(\delta_{ik} + \partial_i w_{e_k}\right) - \sum_{i=1}^d
    a^{per}_{ij} \partial_i\widetilde w_{e_k}.  
\end{equation}
In the case $r<d$,
we simply apply \eqref{eq:B_borne}, finding that $\widetilde B \in L^\infty(\RR^d)$, which implies
\eqref{eq:estimation_B}, since $\nu_r=1$. In the case $r>d$, we have $\nabla \widetilde B \in L^q(\RR^d)$, and we
may apply Morrey's Theorem as we did above for $\widetilde w_p$. This proves \eqref{eq:estimation_B}.
\end{proof}

Collecting the results of Proposition~\ref{pr:defaut1} and Proposition~\ref{pr:defaut2}, we have thus proved the
following Proposition, which in turn implies Theorem~\ref{th:defaut}.
\begin{proposition}\label{pr:3}
  Assume that $r\in [1,+\infty[$, $r\neq d$, and that the coefficient $a$ satisfies \eqref{eq:aper+tildea} and
  \eqref{eq:hyp1}. Then $a$ satisfies Assumptions~\ref{H1} through \ref{H6}, and
  \ref{H7}-\ref{H8}, with $\nu = \nu_r$ defined by \eqref{eq:nu_r}. 
\end{proposition}
\begin{proof}
  It is clear that \eqref{eq:hyp1} implies \ref{H1} and \ref{H2}. As mentioned above, the results of
  \cite{BLLcpde,BLLfutur1} imply that \ref{H3} and \ref{H4} are satisfied. Proposition~\ref{pr:defaut1}
  implies \ref{H7}, and Proposition~\ref{pr:defaut2} implies \ref{H8}. Finally, Lemma~\ref{lm:rq2} implies \ref{H5}
  and \ref{H6}.
\end{proof}

\section*{Acknowledgements} The work of the third author is partially supported by ONR under Grant N00014-15-1-2777  and by EOARD, under Grant FA-9550-17-1-0294. 

  \bibliographystyle{siam}
  \bibliography{article}

\end{document}